\documentclass[12pt]{article}
\usepackage{amsmath,amssymb,amsthm,times,authblk,mathtools}
\usepackage{url}
\usepackage{parallel,enumitem}
\usepackage{rotating}
\usepackage{lscape}
\usepackage{hyperref}
\usepackage{caption}
\usepackage{epstopdf}
\usepackage{graphicx}
\usepackage{etoolbox}
\usepackage{slashbox}

\newtheorem{theorem}{Theorem}
\newtheorem{lemma}{Lemma}

\newtheorem{example}{Example}
\newtheorem{remark}{Remark}

\allowdisplaybreaks 

\title{Improved Bohr radius for $ k$-fold symmetric univalent logharmonic mappings}
\author{Akash Meher$^{a}$, Priyabrat Gochhayat$^{b}$}
\affil{Department of  Mathematics, Sambalpur University, Sambalpur, 768019, Odisha, India
\newline $^{a}$\textit{E-mail}: \textup{meherakash248@gmail.com, meherakash@suniv.ac.in}, \\ $^{b}$\textit{E-mail}: \textup{pgochhayat@gmail.com}}

\begin{document}
\maketitle
\begin{abstract}
We study the $k$-fold symmetric starlike univalent logharmonic mappings of the form $f(z)=zh(z)\overline{g(z)}$ in the unit disk $\mathbb{D}:= \lbrace z \in \mathbb{C}: |z|<1 \rbrace$ with several examples, where $h(z)=\exp \left(\sum_{n=1}^{\infty}a_{nk}z^{nk}\right)$ and $g(z)=\exp \left(\sum_{n=1}^{\infty}b_{nk}z^{nk}\right)$ are analytic in $\mathbb{D}.$ The distortion bounds of these functions are obtained, which give area bounds. Improved Bohr radii for this family are calculated. We also introduce the pre-Schwarzian and Schwarzian derivatives of logharmonic mappings that vanishes at the origin.
\end{abstract}
\noindent {\it{2020 Mathematics Subject Classification:}} 30A10, 30C35, 30C45.

\vspace{.2in} \noindent {\it {Keywords:}} Logharmonic mappings, $k$-fold symmetric mapping, distortion bound, improved Bohr radius, Schwarzian derivative

\section{Introduction}\label{introduction}

Let $\mathcal{A}(\mathbb{D})$ be the class of analytic functions defined in the open unit disk $\mathbb{D}:=\lbrace z \in \mathbb{C}:|z|<1\rbrace$.  A complex-valued function $f$ is said to be harmonic, if both $Re\lbrace f \rbrace$ and $Im \lbrace f\rbrace$ are real harmonic. In other words, harmonic functions $f$ are the solutions of $\Delta f=0$, where $\Delta$ is the Laplacian operator and defined as 

\begin{equation*}
\Delta=4\dfrac{\partial^2}{\partial z \partial \overline{z}}=\dfrac{\partial^2}{\partial x^{2}}+\dfrac{\partial^2}{\partial y^2}.
\end{equation*}
Every harmonic mapping $f$ has crucial property, that is, it has canonical decomposition $f=h+\overline{g}$, where $h,g \in \mathcal{A}(\mathbb{D})$ and respectively known as analytic and co-analytic part of $f$. Denote $\mathcal{H}(\mathbb{D}),$ the class of all complex-valued harmonic mappings defined in $\mathbb{D}$.
%It is important to notice that, harmonic mappings have a strong connection with the theory of minimal surfaces (cf. \cite{nit1989, oss2014}).

A mapping $f$ defined in $\mathbb{D}$, is logharmonic, if $\log (f) \in \mathcal{H}(\mathbb{D})$. Alternatively, the logharmonic mappings are the solutions of the non-linear elliptic partial differential equation 
\begin{equation}\label{s1e1}
\dfrac{\overline{f}_{\overline{z}}}{\overline{f}}=\omega \dfrac{f_{z}}{f},
\end{equation}
where $\omega \in \mathcal{A}(\mathbb{D})$, $|\omega|<1$ and is known as second dilatation of $f$. Note that if $f_{1}$ and $f_{2}$ are two logharmonic functions with respect to $\omega$, then $f_{1}f_{2}$ is logharmonic with respect to same $\omega$, and $f_{1}/f_{2}$ is logharmonic (provided $f_{2} \neq 0$). The logharmonicity preserves pre-composition with a conformal mapping, whereas it is not always true for a post composition. Furthermore, logharmonicity is not invariant under translation and inversion. Every non-constant logharmonic mapping is quasiregular, therefore it is continuous and open. The logharmonic mapping $f$ is also sense preserving as its Jacobian  
\begin{equation}\label{s1e2}
J_{f}(z)= |f_{z}(z)|^{2}-|f_{\overline{z}}(z)|^{2}= |f_{z}|^{2} \left(1-|\omega|(z)|^{2}\right), \qquad z \in \mathbb{D}.
\end{equation}
is positive. The salient properties like the modified Liouville's theorem, the maximum principle, the identity principle, and the argument principle hold true for logharmonic mappings (cf. \cite{abdali2012}).

A non-constant logharmonic mapping $f$ defined in $\mathbb{D}$ bear the representation \cite{abdbsh1988} (also see \cite{abdali2012})
\begin{equation*}\label{ls1e1}
f(z)=z^{m} |z|^{2 \beta m} h(z) \overline{g(z)}, \qquad z \in \mathbb{D},
\end{equation*}
where $m$ is a non-negative integer, $Re(\beta)>1/2$ and $h,g \in \mathcal{A}(\mathbb{D})$ satisfying $h(0)\neq 0$ and $g(0)=1$. Here, $\omega(0)$ is the only factor that determines the exponent $\beta$ in the following way 
\begin{equation*}
\beta=\overline{\omega(0)}\dfrac{1+\omega(0)}{1-|\omega(0)|^2}.
\end{equation*}
For $m=0$, $f$ is non-vanishing at origin and vice versa, and this type of mapping admits the representation 
\begin{equation*}
f(z)=h(z)\overline{g(z)}, \qquad z \in \mathbb{D},
\end{equation*}
where $h, g \in \mathcal{A}(\mathbb{D})$. The extensive studies of non-vanishing logharmonic mappings are found in \cite{abd1996, abd2002, maopon2013, ozkpot2013}. If $f$ is non-constant univalent logharmonic mapping in $\mathbb{D}$ such that $f(0)=0$ and $f(z)\neq 0$ elsewhere, then $f$ has the form 
\begin{equation*}
f(z)=z|z|^{2\beta}h(z)\overline{g(z)}, \qquad z\in \mathbb{D},
\end{equation*}
where $Re(\beta)>1/2$, and $h,g \in \mathcal{A}(\mathbb{D})$ with $0 \notin (hg)(\mathbb{D})$ and $g(0)=1$. This type of mapping is widely studied in \cite{abd1996, abdabu2006, abdbsh1988, abdhen1987, aliabd2016, aliagh2022, liupon2018}. For more information on univalent logharmonic mappings, we refer to the review article \cite{abdali2012}.

Denote $S_{LH}$, the class of univalent logharmonic mappings $f$ in $\mathbb{D}$ of the form 
\begin{equation}\label{s1e3}
f(z)=zh(z)\overline{g(z)}, \qquad z \in \mathbb{D},
\end{equation} 
where $h,g \in \mathcal{A}(\mathbb{D})$ with 
\begin{equation*}
h(z)=\exp \left( \sum_{n=1}^{\infty} a_{n}z^n\right)~~\text{and} ~~ g(z)=\exp \left( \sum_{n=1}^{\infty}b_{n}z^n \right).
\end{equation*}
A function $f \in S_{LH}$ is said to be starlike of order $\alpha$ if 
\begin{equation*}
\dfrac{\partial}{\partial \theta}\left( \arg f(re^{i\theta})\right)=Re\left(\dfrac{Df(z)}{f(z)}\right) > \alpha,
\end{equation*}
for all $z=re^{i\theta} \in \mathbb{D}$, where the operator $D=z\frac{\partial}{\partial z}-\overline{z}\frac{\partial}{\partial \overline{z}}$ and $0 \leq \alpha <1$. This type of function form the class $S_{LH}^{*}(\alpha)$ (cf. \cite{abdabu2006, liupon2018}). For $\alpha =0$, the class becomes $S_{LH}^{*}$, the class of starlike univalent logharmonic mappings (cf. \cite{abdhen1987, aliabd2016}). This paper is mainly dealing with the functions which are $k$-fold symmetric and in the class $S_{LH}^{*}(\alpha)$. The detailed explanations of the class are in Section \ref{symmetric}.

\subsection{Bohr radius}

In \cite{boh1914}, Bohr described the size of the moduli of terms of the power series of a bounded analytic function, which is now called as Bohr phenomenon. For $f(z)=\sum_{n=0}^{\infty}a_{n}z^{n}\in \mathcal{A}(\mathbb{D})$ with $|f(z)|<1$ in $\mathbb{D}$, Bohr obtained the inequality
\begin{equation}\label{bohr}
\sum_{n=0}^{\infty}|a_{n}||z^n|<1, \qquad z \in \mathbb{D}
\end{equation} 
for $|z|<1/6$. Later, Weiner, Schur and Riesz independently found the sharp value $|z|<1/3$, for which the Bohr inequality \eqref{bohr} holds (cf. \cite{poupop2002, sid1927, tom1962}). The value $1/3$ is known as the classical Bohr radius for the class of analytic function. The inequality \eqref{bohr} can be expressed in terms of the distance formula
\begin{equation*}\label{bohr2}
d \left(\sum_{n=0}^{\infty}|a_n z^n|, |a_0|\right)=\sum_{n=1}^{\infty}|a_n z^n| \leq 1-f(0)=d(f(0), \partial f(\mathbb{D}))
\end{equation*}
for $|z|<1/3$, where $d$ is the Euclidean distance and $\partial \mathbb{D}$ is the boundary of the disk $\mathbb{D}.$ The theory is also represented with regard to hyperbolic metric (see \cite{abuali2013}). 

Following the work due to Bohr \cite{boh1914}, researchers raised the Bohr phenomenon to an active area of research by investigating it in one and several complex variables, and also in different contexts. A connection between Bohr inequality and characterizations of Banach Algebra satisfying von Neumann's inequality led operator algebraists to be more interested in Bohr inequality after Dixon \cite{dix1995} showed it existed. The extension of Bohr's theory to the context of Banach Algebra is found in \cite{bla2010, defgar2003, polsin2004}. Ali et al. \cite{alibar2017} calculated the Bohr radius for odd and even analytic functions and alternating series. For the case of $k$-fold symmetric analytic mappings, the Bohr radius is obtained due to Kayumov and Pinnusamy \cite{kaypon2017, kaypon2018c} by making use of Cauchy-Schwarz inequality and subordination principle. The theory is generalised to concave wedge domain in \cite{abuali2014, alibar2017}. The paper \cite{abuali2014} is also dealing with the linkage between the Bohr phenomenon and differential subordination. Balasubramanian et al. \cite{balcal2006} introduced Bohr inequality for the Dirichlet series. The multidimensional Bohr radius was introduced by Boas and Khavinson \cite{baokha1997} with a conclusion that the radius decreases to zero in accordance to increase in dimension. We refer \cite{defgar2004, deffre2006, deffre2011, defgar2003} for more information about the Bohr phenomenon on the multidimensional case and Banach space theory. Abu-Muhanna and Gunatillake \cite{abugun2012} obtained the Bohr radius for weighted Hardy--Hilbert space and concluded that the Bohr radius of classical Hilbert space cannot be obtained. The Bohr phenomenon on multidimensional weighted Hardy-Hilbert space is derived in \cite{poupop2002}. In \cite{benkor2004}, the authors show the existence of the theory for the case of Hardy space. Abu-Muhanna \cite{abu2010} for the first time generalised the theory of the Bohr phenomenon to harmonic mappings, but the theory was not correct for all cases. Kayumov et al. \cite{kaypon2018} gave the actual harmonic extension of classical Bohr inequality and obtained the Bohr radius for the class of locally univalent harmonic mappings, quasiconformal harmonic mappings, analytic and harmonic Bloch space. Kayumov and Ponnusamy \cite{kaypon2018a} improved the classical Bohr inequality with new four different formulations. The results are further investigated and improved in the recent review article due to Ismagilov et al. \cite{ismkay2021}. The harmonic extension of improved Bohr's radius for locally univalent harmonic mapping is calculated by Evdoridis et al. \cite{evdpon2019}. The logharmonic analogue of classical Bohr inequality is shown by Ali et al. \cite{aliabd2016} and they calculated the Bohr radius for the functions in the class $S_{LH}^{*}$. The improved version of Bohr estimate for the functions class $S_{LH}^{*}$ are obtained in \cite{ahmall2021b}. In this paper, we compute the improved Bohr radius and Bohr type inequalities for the functions in the class $S_{LH}^{k*}(\alpha)$, the class of univalent $k$-fold symmetric mappings which are logharmonic starlike of order $\alpha$.

The paper is organised as follows: In Section \ref{symmetric}, we describe the class $S_{LH}^{k*}(\alpha)$ by considering various examples. Using the distortion bounds of the functions in $S_{LH}^{k*}(\alpha)$, we calculate the area bounds in Section \ref{distortion}. We present improved Bohr radius and Bohr type inequalities with their numerical illustration in Section \ref{improved}. The introduction of the Pre-Schwarzian and  Schwarzian derivative for logharmonic mappings of the form $f(z)=zh(z)\overline{g(z)}$, where $h(z), g(z) \in \mathcal{A}(\mathbb{D})$ is in Section \ref{schwarzian}.

\section{\boldmath{$k$}-fold symmetric starlike logharmonic mappings}\label{symmetric}
For a positive integer $k$, an analytic function $f$ defined in $\mathbb{D}$ is $k$-fold symmetric if $f(e^{\frac{2\pi i}{k}}z)=e^{\frac{2 \pi i}{k}}f(z)$ and has the Taylor-Maclaurin series representation
\begin{equation}\label{ffold}
f(z)=\sum_{n=0}^{\infty}A_{nk+1}z^{nk+1}, \qquad z \in \mathbb{D}.
\end{equation}
Conversely, any function $f$ with the power series representation \eqref{ffold} is $k$-fold symmetric inside the domain of convergence of the series (cf. \cite[pp. 18]{goo1983}, \cite{sripra2022}). It is important to note that not all $k$-fold symmetric mappings are univalent. Denote $\mathcal{S}^{k}$ the class of all $k$-fold symmetric univalent analytic functions. The class of univalent odd analytic functions is obtained for $k=2$. The functions which are univalent analytic, $k$-fold symmetric and starlike of order $\alpha$, constitute the class $\mathcal{S}^{k*}(\alpha)$, where $0 \leq \alpha <1$.  

A mapping $f$ is said to be $k$-fold symmetric logharmonic if it is $k$-fold symmetric and solution of \eqref{s1e1} with respect to some $\omega$. Let $S_{LH}^{k*}(\alpha)$ be the class of univalent $k$-fold symmetric mappings which are logharmonic starlike of order $\alpha$ defined in $\mathbb{D}$ with the representation $f(z)=zh(z)\overline{g(z)}$, where
\begin{equation*}
h(z)=\exp \left( \sum_{n=1}^{\infty}a_{nk}z^{nk}\right) ~~ \text{and} ~~ g(z)=\exp \left( \sum_{n=1}^{\infty}b_{nk}z^{nk} \right), \qquad z \in \mathbb{D}.
\end{equation*}
When $g \equiv 1$ in $\mathbb{D}$, the function $f$ is in $\mathcal{S}^{k*}(\alpha)$. The following lemma gives the bridge between the classes $\mathcal{S}^{k*}(\alpha)$ and $S^{k*}_{LH}(\alpha)$ which is a consequence of Theorem-2.1 of \cite{abdabu2006}
\begin{lemma}\label{s1l1}
 A function $f(z)=zh(z)\overline{g(z)} \in S^{k*}_{LH}(\alpha)$ if and only if $\phi(z)=\frac{z h(z)}{g(z)} \in \mathcal{S}^{k*}(\alpha)$ in $\mathbb{D}$. 
\end{lemma}
Our first observation on the class $S_{LH}^{k*}(\alpha)$ is about the logarithm convexity property.
\begin{theorem}
The function class $S_{LH}^{k*}(\alpha)$ is closed under logarithmic convex combination.
\end{theorem}
\begin{proof}
Suppose $f_{1}, f_{2} \in S_{LH}^{k*}(\alpha)$ and are the solutions of \eqref{s1e1} with respect to same $\omega$. For $\gamma \in (0, 1)$, define $f(z)=(f_{1}(z))^{\gamma}(f_{2}(z))^{1-\gamma}$ in $\mathbb{D}$. Then, it is simple to see that $f(e^{\frac{2\pi i}{k}}z)=e^{\frac{2 \pi i}{k}}f(z)$ for some positive integer $k$. Using the property of logharmonic mappings, $f$ is a solution of \eqref{s1e1} with respect to same $\omega$ as for $f_{1}$ and $f_{2}$ and hence $f$ is in $S_{LH}^{k*}(\alpha)$. This completes the proof.
\end{proof}
%\newpage
%
%
%A mapping $f$ is said to be log-harmonic, if it is a solution of the non-linear elliptic partial differential equation
%\begin{equation}\label{s1e1}
%\dfrac{\overline{f}_{\overline{z}}}{\overline{f}}=\omega \dfrac{f_{z}}{f},
%\end{equation}
%where $\omega$ is analytic and $|\omega|<1$ in $\mathbb{D}$. The Jacobian for these mappings is defined as
%\begin{equation}\label{s1e2}
%J_{f}(z)= |f_{z}(z)|^{2}-|f_{\overline{z}}(z)|^{2}= |f_{z}|^{2} \left(1-|\omega|(z)|^{2}\right).
%\end{equation}
%
%For a positive integer $k$, an analytic function $f$ is $k$-fold symmetric if $f\left(e^{\frac{2\pi i}{k}}z\right)=e^{\frac{2\pi i}{k}}f(z)$ and has the Taylor-Maclaurin series representation 
%$$f(z)=\sum_{n=0}^{\infty}A_{kn+1}z^{kn+1}.$$
%The class $S^{k}_{LH}$ consisting of the $k$-fold symmetric log-harmonic function $f$ of the form 
%\begin{equation}\label{s1e3}
%f(z)=zh(z)\overline{g(z)},
%\end{equation}
%where $h$ and $g$ are analytic in $\mathbb{D}$, and has the representations respectively as 
%\begin{equation}\label{s1e4}
%h(z)=\exp\left( \sum_{n=1}^{\infty}a_{kn}z^{kn} \right) \qquad (k\geq 1) 
%\end{equation}
%and 
%\begin{equation}\label{s1e5}
%g(z)=\exp\left( \sum_{n=1}^{\infty}b_{kn}z^{kn} \right) \qquad (k\geq 1).
%\end{equation}

We would like to point out that not all logharmonic mappings are $k$-fold symmetric and vice versa. For example, the function $\tilde{f}_{1}(z)=z^{3}(\overline{z})^{2}$ is $k$-fold symmetric, but is not a solution of \eqref{s1e1}. On the other hand, $\tilde{f}_{2}(z)=z^3\overline{z}$ is a solution of \eqref{s1e1} with $\omega(z)=1/3$, but is not $k$-fold symmetric. Now we present some examples of the $k$-fold symmetric logharmonic mappings as follows:
% The following examples clarify it:
%\begin{example}\normalfont
%The function $f(z)=z^{3}\overline{z}^{2}$ is $k$-fold symmetric but it is not a solution of \eqref{s1e1}, therefore it is not log-harmonic.
%\end{example}
%\begin{example}\normalfont
%The function $f(z)=z^{3}\overline{z}$ is log-harmonic as it is a solution of \eqref{s1e1} with $\omega(z)=1/3$, but it is not $k$-fold symmetric.
%\end{example}

\begin{example}\normalfont
Consider the function $\tilde{f}_{3}(z)=z^{k+1}(\overline{z})^{k}$ which is not harmonic in $\mathbb{D^{*}}=\mathbb{D}\setminus \{0\}$ as $\tilde{f_3}_{z\overline{z}}=k(k+1)|z|^{k-1}z \neq 0, \forall z\in \mathbb{D}^{*}$. However, observe that $\tilde{f}_{3}$ is $k$-fold symmetric and also a solution of \eqref{s1e1} with $\omega (z)=k/(k+1)$. Furthermore, $\frac{D\tilde{f_3}(z)}{\tilde{f_3}(z)}=1$ showing that $f$ is univalent, starlike and sense preserving. Therefore, $\tilde{f_3}$ is $k$-fold symmetric starlike sense preserving univalent logharmonic mapping in $\mathbb{D^{*}}.$
 \end{example}
\begin{example}\normalfont
Some simple calculations show that the function $\tilde{f}_{4}(z)=\dfrac{z(1-\overline{z}^{k})}{1-z^{k}},$ where $k$ is a positive integer, is $k$-fold symmetric and a solution of \eqref{s1e1} with $$\omega(z)= \dfrac{-kz^{k}}{1+(k-1)z^{k}}.$$ Therefore, $\tilde{f_4}$ is $k$-fold symmetric sense preserving logharmonic mapping in $\mathbb{D}$. Further, $$\frac{D\tilde{f_4}(z)}{\tilde{f_4}(z)}=\dfrac{1+(k-1)z^{k}}{1-z^{k}}+\dfrac{k\overline{z}^{k}}{1-\overline{z}^{k}}$$ imply that the radius of starlikeness of $\tilde{f}_{4}$ is the unique root in $(0,1)$ of $1+(1-2k)r^{2k}-2(k-1)r^{k} = 0$. 
\end{example}

\begin{example}\normalfont
The function $\tilde{f}_{5}(z)=\dfrac{z^{2}\overline{z}}{(1-z^{k})^{2(1-\alpha)}},~ 0 \leq \alpha <1.$ Here, $\tilde{f}_{5}$ is $k$-fold symmetric and also a solution of \eqref{s1e1} with $$\omega (z)=\dfrac{1-z^{k}}{2(1-z^{k})+2k(1-\alpha)z^{k}}.$$
This shows that $\tilde{f_5}$ is a member of $k$-fold symmetric logharmonic mappings in $\mathbb{D}$. Note that 
\begin{equation*}
\dfrac{D \tilde{f_5}(z)}{\tilde{f_5}(z)}=1+\dfrac{2k(1-\alpha)z^k}{1-z^k}.                             
\end{equation*} This implies that the function $\tilde{f}_{5}$ is starlike of order $\alpha$ within the radius $|z|<r^{*}$, where $r^{*}$ is the unique root in $(0,1)$ of $1+(1-2k)r^{2k}-2(k-1)r^{k} = 0$.
\end{example}
Thus, in general we have the following:
\begin{theorem}
    Every function $f \in S_{LH}^{k*}(\alpha)$ is starlike of order $\alpha$ in $|z|<R,$ where $R$ is the unique root in $(0,1)$ of $(1+2\alpha)r^{2k}-(6-2\alpha)r^{k}+1=0.$
\begin{figure}
\caption{Images of the unit disk under the mappings $\tilde{f_4}$ and $\tilde{f_5}$}
\begin{minipage}[h]{2in}
\includegraphics[width=5cm, height=4.2cm]{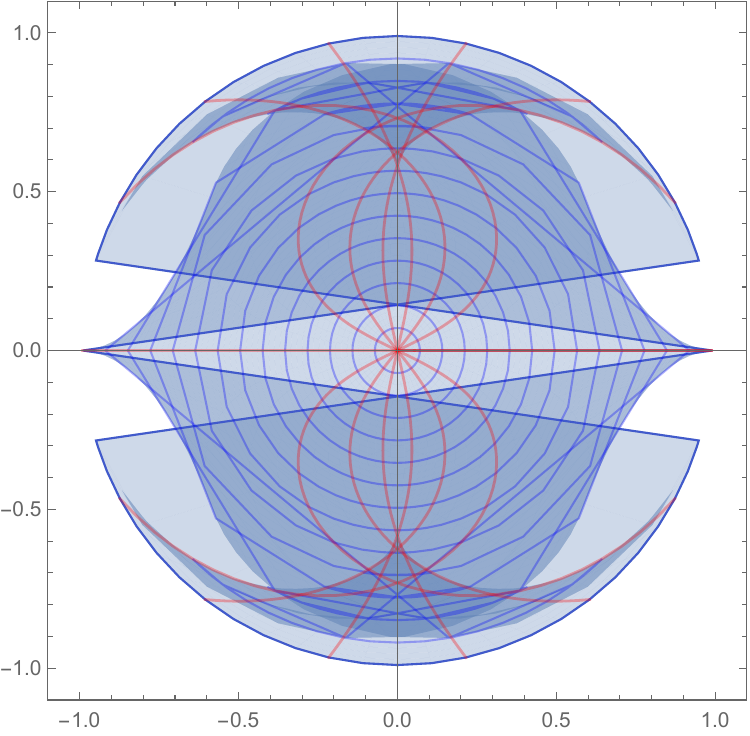}
\captionof*{figure}{$2$-fold $\tilde{f}_{4}$}
\end{minipage}
%\vspace{0.2cm}
\hspace{0.5cm}
\begin{minipage}[h]{2in}
\includegraphics[width=4.7cm, height=4.2cm]{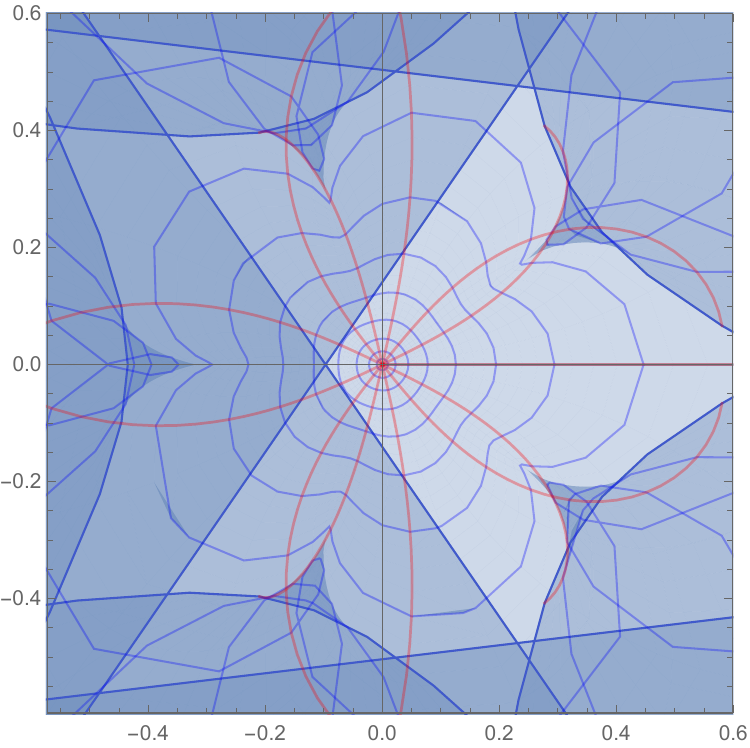}
\captionof*{figure}{$5$-fold $\tilde{f}_{5}$ when $\alpha =0.4$}
\end{minipage}
%\vspace{0.2cm}
\hspace{0.5cm}
\begin{minipage}[h]{2in}
\includegraphics[width=4.7cm, height=4.2cm]{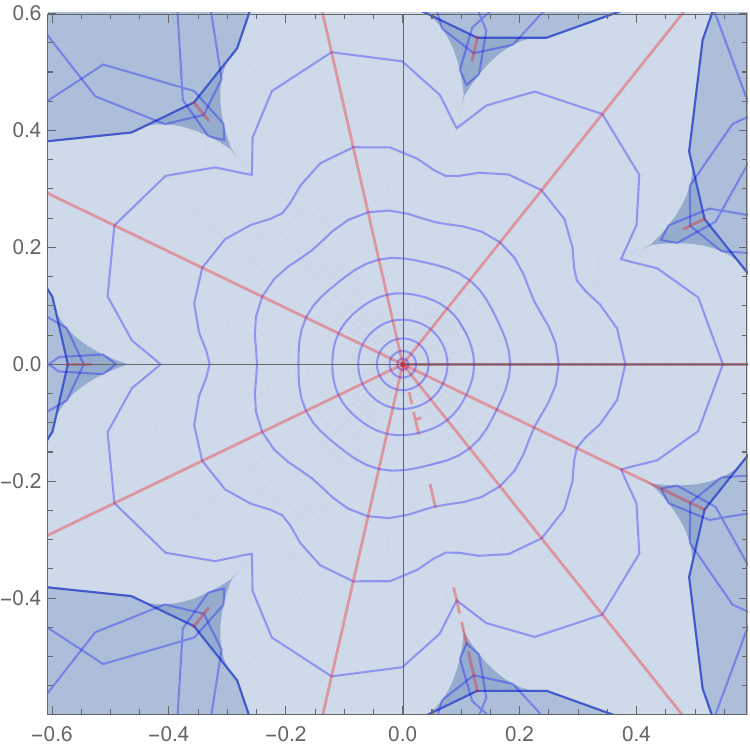}
\captionof*{figure}{$7$-fold $\tilde{f}_{5}$ when $\alpha = 0.2$}
\end{minipage}
\vspace{0.5cm}
\end{figure}

\end{theorem}
Recall the function $f_{\alpha}$ (cf. \cite{aliagh2022}) defined by 
\begin{equation}\label{s1l3e1}
f_{\alpha}(z)=zh_{\alpha}(z)\overline{g_{\alpha}(z)}=\dfrac{z}{(1-z^{k})^{\frac{1}{k}}}\dfrac{1}{(1-\overline{z}^{k})^{\frac{2\alpha -1}{k}}} \exp\left(\dfrac{(1-\alpha)}{k}Re \left(\dfrac{4z^{k}}{1-z^{k}}\right) \right), \qquad z \in \mathbb{D},
\end{equation}
where the analytic functions $h_{\alpha}$ and $g_{\alpha}$ are of the form
\begin{equation}\label{h}
h_{\alpha}(z)=\dfrac{1}{(1-z^k)^{\frac{1}{k}}}\exp \left( \dfrac{2(1-\alpha)z^k}{k(1-z^k)}\right), \qquad z\in \mathbb{D}
\end{equation}
and
\begin{equation}\label{g}
g_{\alpha}(z)=\dfrac{1}{(1-z^k)^{\frac{2\alpha -1}{k}}}\exp \left(\dfrac{2(1-\alpha)z^k}{k(1-z^k)}\right), \qquad z \in \mathbb{D}.
\end{equation}
It is important to note that, the characteristics of $f_{\alpha}$ in the family $S_{LH}^{k*}(\alpha)$ are same as that of the Koebe function in the family $\mathcal{S}$, analytic univalent functions. In other words, the function$f_{\alpha}$ is the $k$-fold symmetric logharmonic Koebe function, which plays extremal role in the calculation of coefficient bounds, growth and covering theorem for the family $S_{LH}^{k*}(\alpha)$ (cf. \cite{aliagh2022}). We examine that $f_{\alpha}$ also provides sharpness to the distortion and area bounds.

\begin{figure}
\caption{Images of the unit disk under the $k$-fold symmetric logharmonic Koebe function $f_\alpha$.}
\begin{minipage}[h]{2in}
\includegraphics[width=5cm, height=4.2cm]{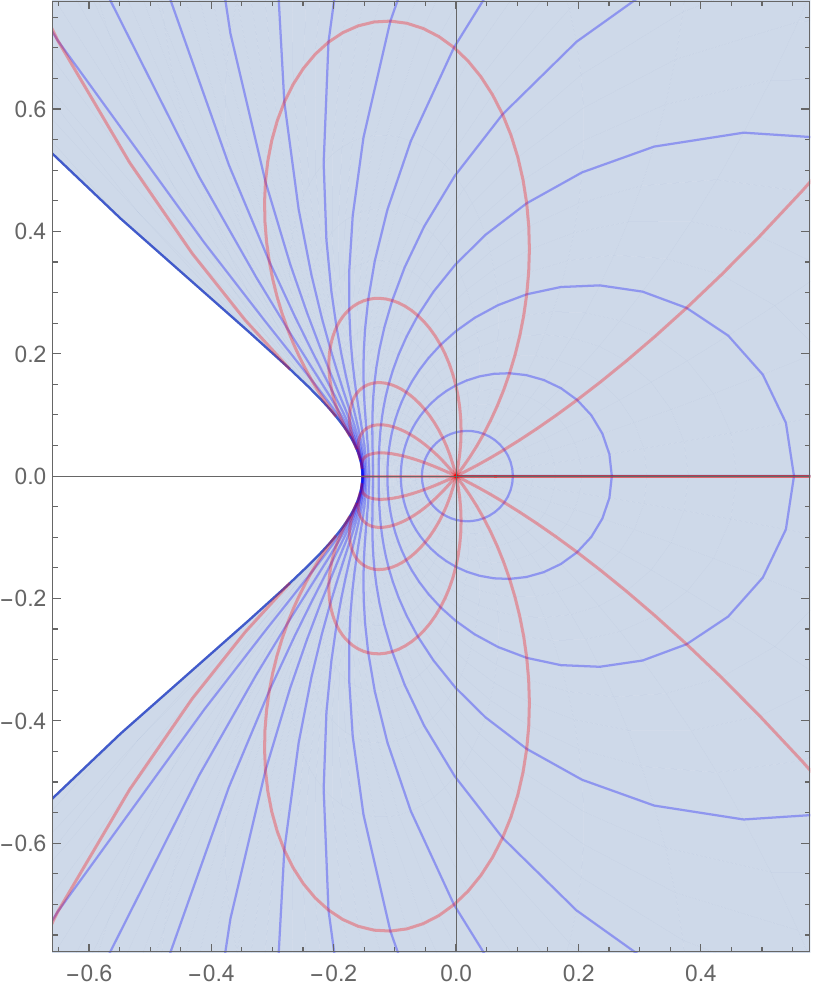}
\captionof*{figure}{For $k=1$, $\alpha =0.2$}
\end{minipage}
%\vspace{1cm}
\hspace{0.5cm}
\begin{minipage}[h]{2in}
\includegraphics[width=5cm, height=4.2cm]{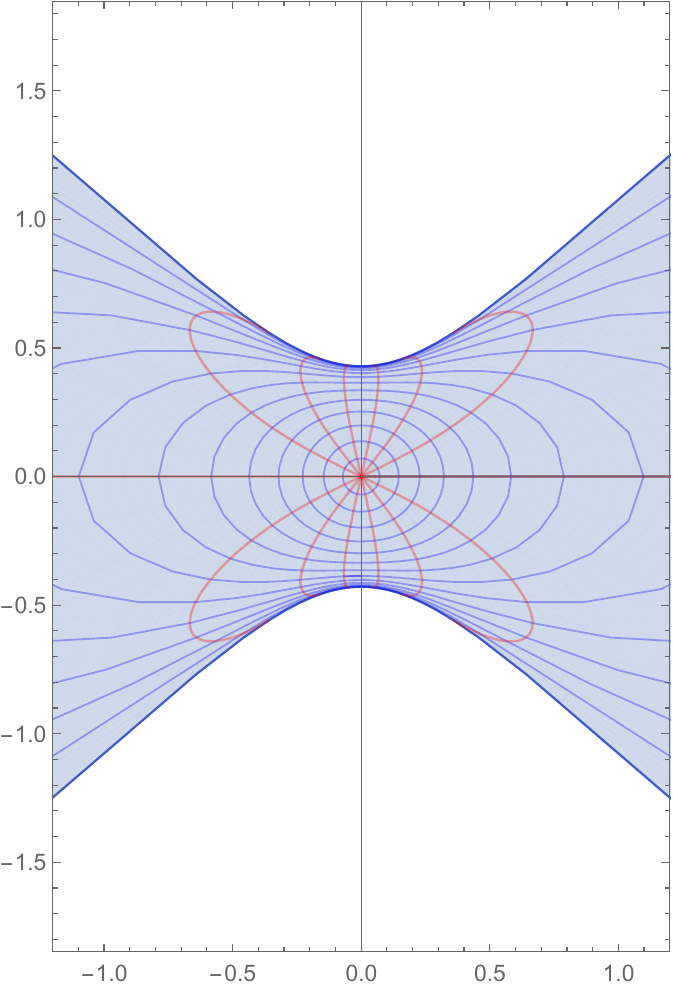}
\captionof*{figure}{For $k=2$, $\alpha =0.5$}
\end{minipage}
%\vspace{0.2cm}
\hspace{0.5cm}
\begin{minipage}[h]{2in}
\includegraphics[width=5cm, height=4.2cm]{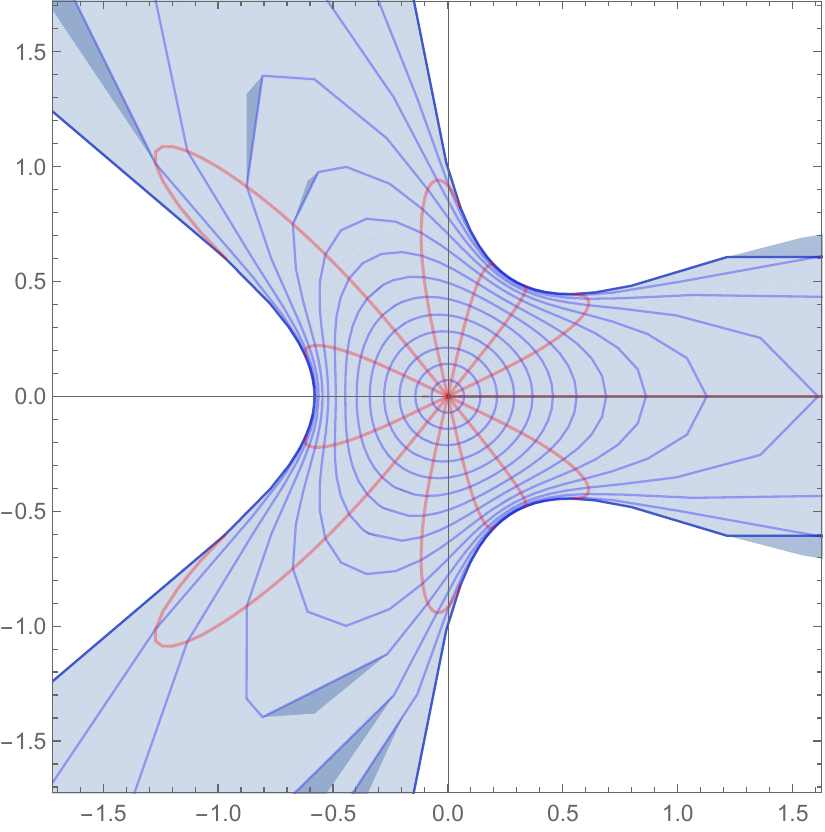}
\captionof*{figure}{For $k=3$, $\alpha =0.6$}
\end{minipage}
\vspace{0.7cm} \\
\begin{minipage}[h]{2in}
\includegraphics[width=5cm, height=4.2cm]{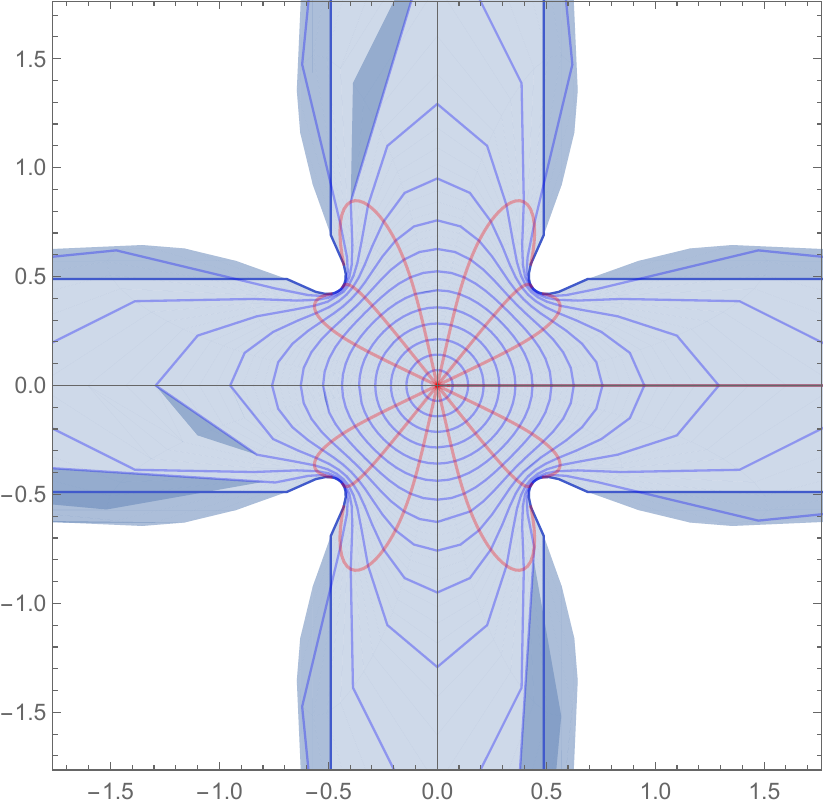}
\captionof*{figure}{For $k=4$, $\alpha =0.2$}
\end{minipage}
%\vspace{0.2cm}
\hspace{0.5cm}
\begin{minipage}[h]{2in}
\includegraphics[width=5cm, height=4.2cm]{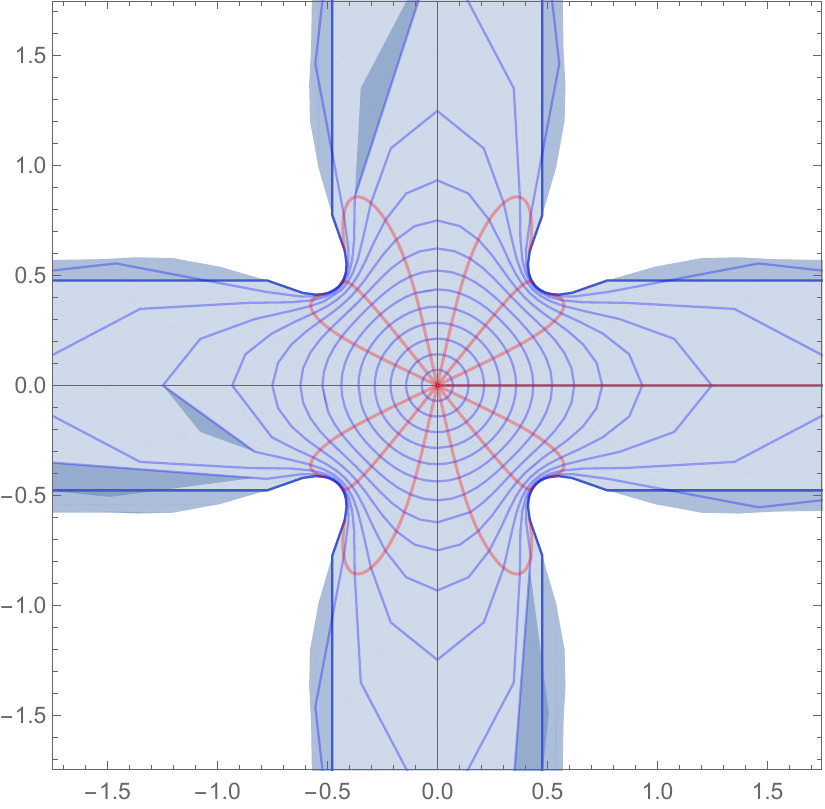}
\captionof*{figure}{For $k=4$, $\alpha =0.3$}
\end{minipage}
%\vspace{0.2cm}
\hspace{0.5cm}
\begin{minipage}[h]{2in}
\includegraphics[width=5cm, height=4.2cm]{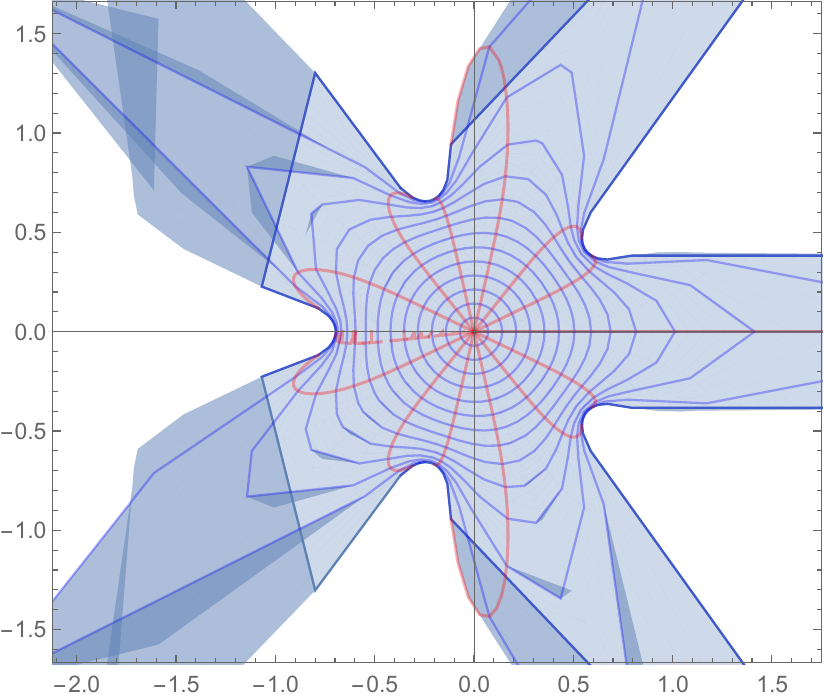}
\captionof*{figure}{For $k=5$, $\alpha =0.3$}
\end{minipage}
\vspace{0.7cm} \\
\begin{minipage}[h]{2in}
\includegraphics[width=5cm, height=4.2cm]{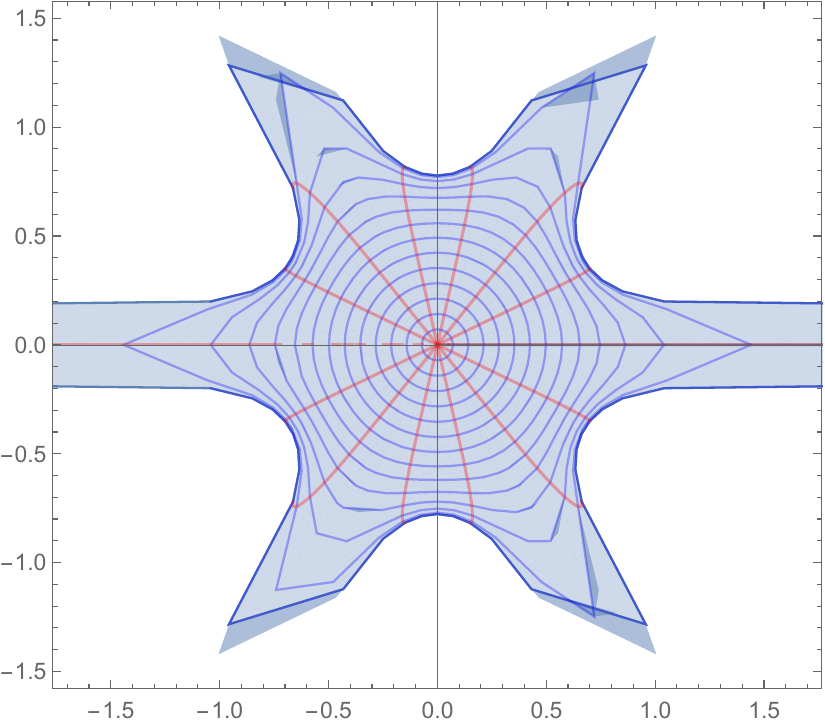}
\captionof*{figure}{For $k=6$, $\alpha =0.8$}
\end{minipage}
%\vspace{1cm}
\hspace{0.5cm}
\begin{minipage}[h]{2in}
\includegraphics[width=5cm, height=4.2cm]{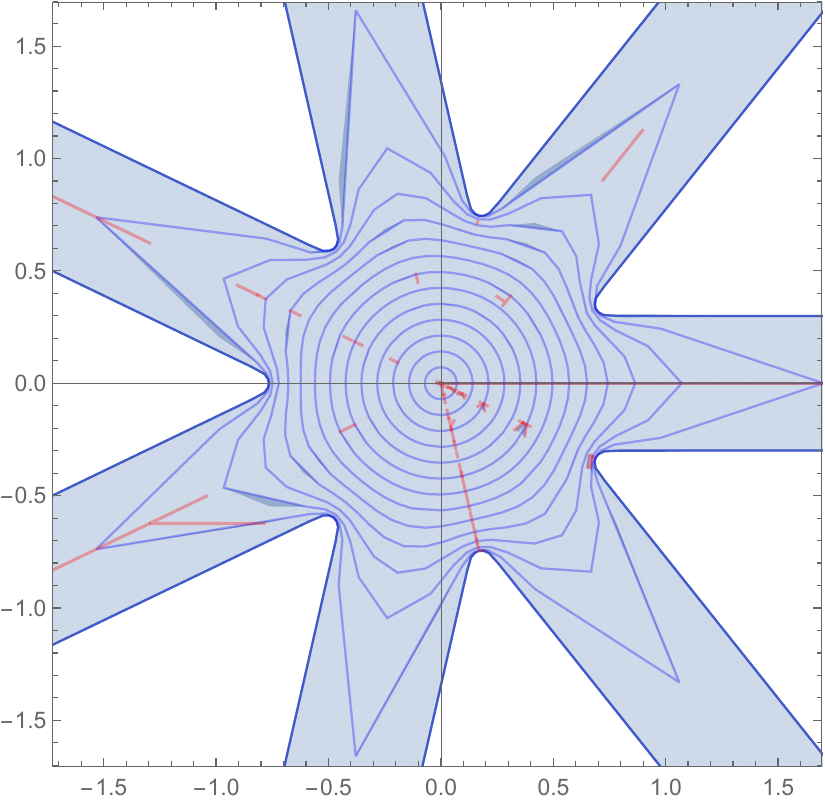}
\captionof*{figure}{For $k=7$, $\alpha =0.2$}
\end{minipage}
%\vspace{0.2cm}
\hspace{0.5cm}
\begin{minipage}[h]{2in}
\includegraphics[width=5cm, height=4.2cm]{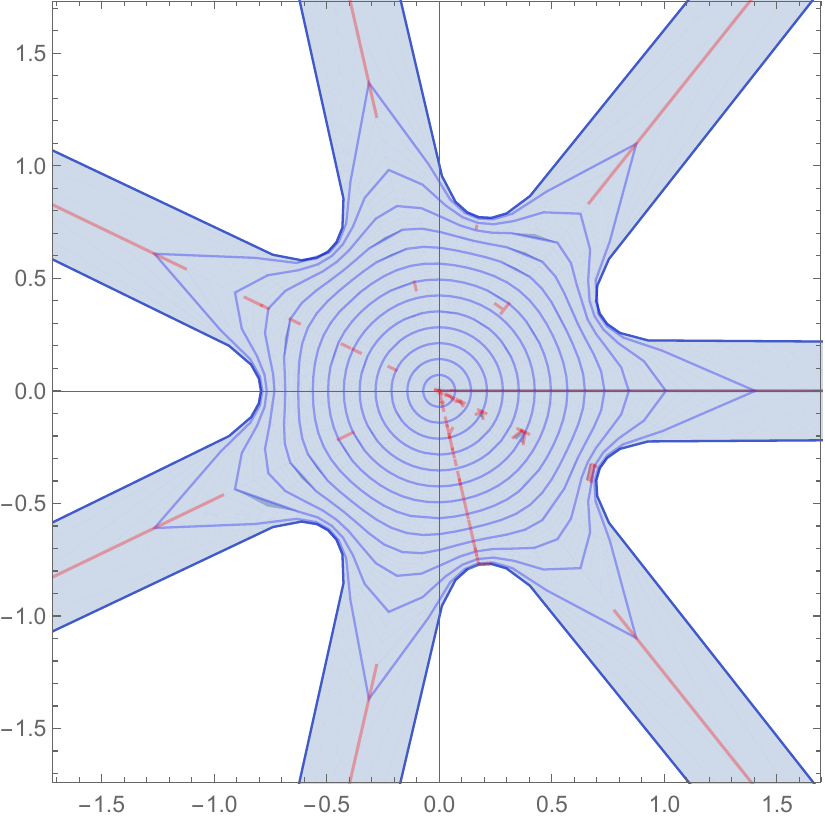}
\captionof*{figure}{For $k=7$, $\alpha =0.6$}
\end{minipage}
\vspace{0.7cm} \\
\begin{minipage}[h]{2in}
\includegraphics[width=5cm, height=4.2cm]{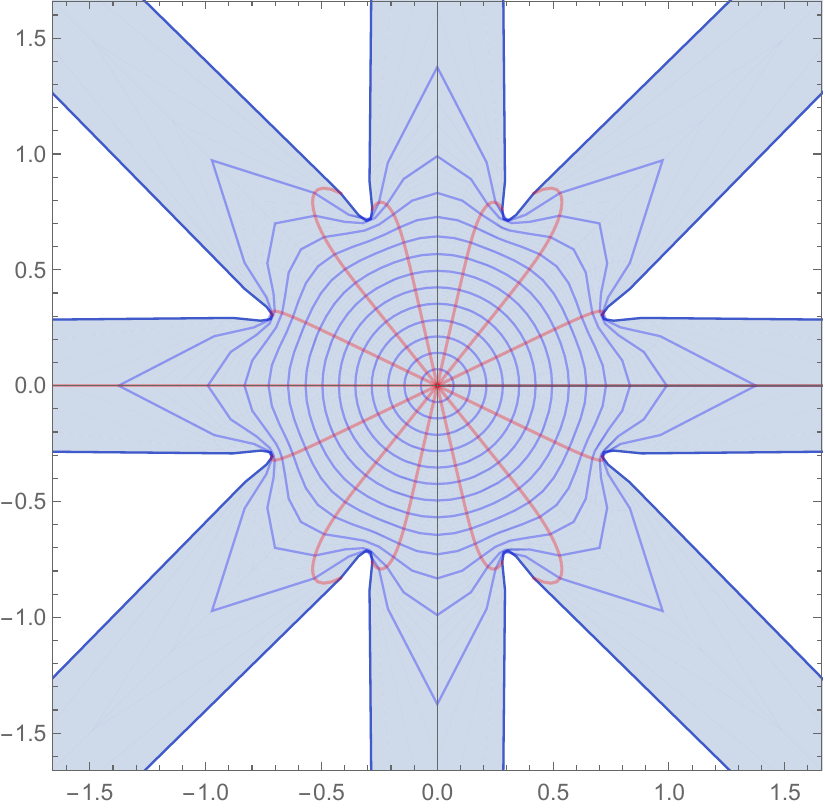}
\captionof*{figure}{For $k=8$, $\alpha =0.5$}
\end{minipage}
%\vspace{1cm}
\hspace{0.5cm}
\begin{minipage}[h]{2in}
\includegraphics[width=5cm, height=4.2cm]{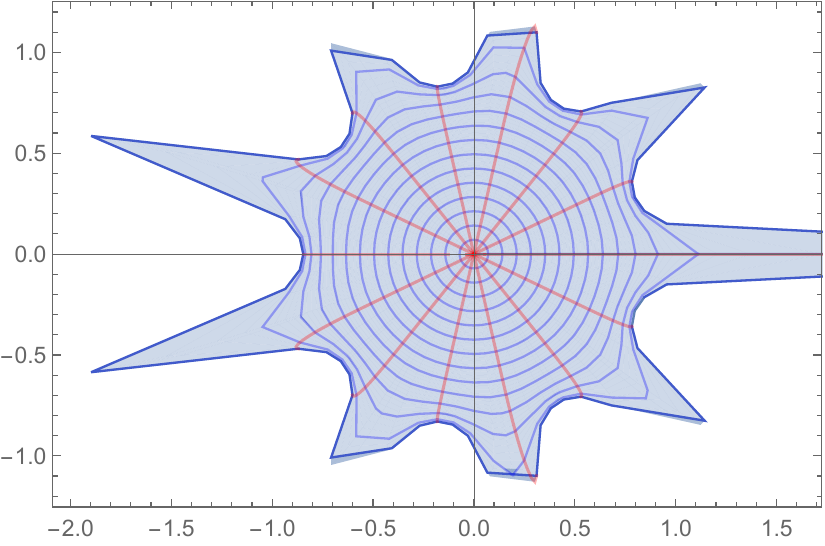}
\captionof*{figure}{For $k=9$, $\alpha =0.8$}
\end{minipage}
%\vspace{0.2cm}
\hspace{0.5cm}
\begin{minipage}[h]{2in}
\includegraphics[width=5cm, height=4.2cm]{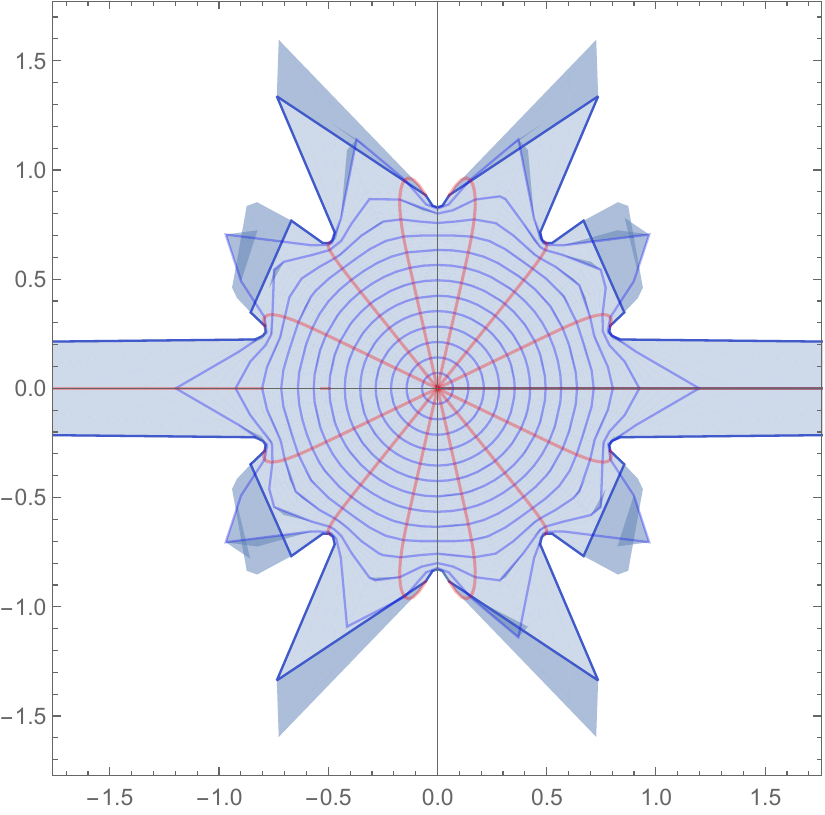}
\captionof*{figure}{For $k=10$, $\alpha =0.2$}
\end{minipage}
\vspace{0.7cm}
\end{figure}

%Let $h_{\alpha}$ and $g_{\alpha}$ be two functions in the class $\mathcal{S}^{k*}(\alpha)$ with
%\begin{equation}\label{h}
%h_{\alpha}(z)=\dfrac{1}{(1-z^k)^{\frac{1}{k}}}\exp \left( \dfrac{2(1-\alpha)z^k}{k(1-z^k)}\right), \qquad z\in \mathbb{D}
%\end{equation}
%and
%\begin{equation}\label{g}
%g_{\alpha}(z)=\dfrac{1}{(1-z^k)^{\frac{2\alpha -1}{k}}}\exp \left(\dfrac{2(1-\alpha)z^k}{k(1-z^k)}\right), \qquad z \in \mathbb{D}.
%\end{equation}
%Then, the function $f_{\alpha}$ of the form 
%\begin{equation}\label{s1l3e1}
%f_{\alpha}(z)=zh_{\alpha}(z)\overline{g_{\alpha}(z)}=\dfrac{z}{(1-z^{k})^{\frac{1}{k}}}\dfrac{1}{(1-\overline{z}^{k})^{\frac{2\alpha -1}{k}}} \exp\left(\dfrac{(1-\alpha)}{k}Re \left(\dfrac{4z^{k}}{1-z^{k}}\right) \right), \qquad z \in \mathbb{D}
%\end{equation}
%plays major role in the class $S_{LH}^{k*}(\alpha)$, much like the role of the Koebe function in the class of analytic univalent functions $\mathcal{S}$. In other word, the function $f_{\alpha}$ is the $k$-fold symmetric logharmonic Koebe function, which plays extremal role in the calculation of coefficient, growth and distortion bounds. 

%  The following lemma gives the bridge between $\mathcal{S}^{*}(\alpha)$ (The class of starlike univalent analytic functions of order $\alpha$) and the class $S^{k*}_{LH}(\alpha)$ which is a consequence of Theorem-2.1 of \cite{abdabu2006}
%\begin{lemma}\label{s1l1}
% Let $f(z)=zh(z)\overline{g(z)}$ be a log-harmonic mapping in $\mathbb{D}$. Then $f \in S^{k*}_{LH}(\alpha)$ if and only if $f \in \mathcal{S}^{*}(\alpha)$. 
%\end{lemma}
\section{Distortion and area bounds}\label{distortion}
To study the geometric properties of $k$-fold symmetric function is not so easy as to calculate extremal functions for classical problem is quite difficult. However, in this paper attempt has been made to derive sharp distortion and area bounds for the functions in the family $S_{LH}^{k*}(\alpha)$. Also, by making use of these results along with growth and coefficient bounds, improved version of Bohr phenomenon are studied and show that the $k$-fold symmetric logharmonic Koebe function plays a crucial role in these aforementioned problems.

 Let $f_{1}, f_{2} \in \mathcal{A}(\mathbb{D})$, then $f_{1}$ is subordinate to $f_{2}$, denoted by $f_{1}\prec f_{2}$, if $f_{1}(z)=f_{2}(\psi (z))$ for all $z\in \mathbb{D}$, where $\psi : \mathbb{D} \rightarrow \mathbb{D}$ is an analytic function such that $\psi (0)=0$ and $| \psi(z) |<1$. This type of function $\psi$ is known as the Schwarz function. Denote $\mathcal{P}$ the class of Carath\'{e}odory functions. In other words, an analytic function $p$ is in $\mathcal{P}$ if $Re(p)>0$ and $p(0)=1$. The class $\mathcal{P}$ and the Schwarz functions are closely related. A function $p \in \mathcal{P}$ if  and only if there is a Schwarz function $\psi$ such that
\begin{equation}\label{schwarz}
p(z)= \dfrac{1+\psi (z)}{1-\psi (z)}, \qquad z \in \mathbb{D}.
\end{equation}
 For more information on subordination and Carath\'{e}odory functions, see the books \cite{dur1983, grakoh2003, halmac1984}. The following preliminaries help us to calculate our main results. 

\begin{lemma}[ cf. Corollary 3.6, \cite{halmac1984}]\label{s1l2}
An analytic function $p(z)$ is in the class $\mathcal{P}$ if and only if there is a probability measure $\nu$ on $\partial \mathbb{D}$ such that
\begin{equation}\label{cara}
p(z)=\int_{\partial \mathbb{D}}\dfrac{1+\eta z}{1-\eta z}d\nu (\eta), \qquad z\in \mathbb{D}.
\end{equation}
\end{lemma}
 Some simple calculation in the expressions \eqref{schwarz} and \eqref{cara} give
 \begin{equation}\label{cara2}
 \dfrac{\psi (z)}{1-\psi (z)}=\int_{\partial \mathbb{D}}\dfrac{\eta z}{1-\eta z}d\nu (\eta), \qquad z\in \mathbb{D}.
 \end{equation}

%The next two lemmas give the growth and the coefficients bounds of the function $f \in S^{k*}_{LH}(\alpha)$.

\begin{lemma}[Theorem 2.2, \cite{aliagh2022}]\label{s1l3}
Any function $f(z)=zh(z)\overline{g(z)} \in S^{k*}_{LH}(\alpha)$ with $\omega (0)=0$   satisfy
\begin{itemize}
\item[(i)\ ]  $\dfrac{\exp\left((1-\alpha)\dfrac{-2|z|^{k}}{k(1+|z|^{k})}\right)}{(1+|z|^{k})^{\frac{1}{k}}}\leq |h(z)| \leq  \dfrac{\exp\left((1-\alpha)\dfrac{2|z|^{k}}{k(1-|z|^{k})}\right)}{(1-|z|^{k})^{\frac{1}{k}}}$,
\item[(ii)\ ]  $ \dfrac{\exp\left((1-\alpha)\dfrac{-2|z|^{k}}{k(1+|z|^{k})}\right)}{(1+|z|^{k})^{\frac{2\alpha -1}{k}}}\leq |g(z)| \leq  \dfrac{\exp\left((1-\alpha)\dfrac{2|z|^{k}}{k(1-|z|^{k})}\right)}{(1-|z|^{k})^{\frac{2 \alpha -1}{k}}}$,
\item[(iii)\ ]  $ \dfrac{|z|\exp\left((1-\alpha)\dfrac{-4|z|^{k}}{k(1+|z|^{k})}\right)}{(1+|z|^{k})^{\frac{2\alpha}{k}}}\leq |f(z)| \leq  \dfrac{|z|\exp\left((1-\alpha)\dfrac{4|z|^{k}}{k(1-|z|^{k})}\right)}{(1-|z|^{k})^{\frac{2\alpha}{k}}}.$
\end{itemize}
The inequalities are sharp for the functions of the form $\overline{\lambda}f_{\alpha}(\lambda z),$ $|\lambda|=1$, where $f_{\alpha}$ is of the form \eqref{s1l3e1}.
%\begin{equation}\label{s1l3e1}
%f_{\alpha}(z)=\dfrac{z}{(1-z^{k})^{\frac{1}{k}}}\dfrac{1}{(1-\overline{z}^{k})^{\frac{2\alpha -1}{k}}} \exp\left(\dfrac{(1-\alpha)}{k}Re \left(\dfrac{4z^{k}}{1-z^{k}}\right) \right)
%\end{equation}

\end{lemma}
\begin{lemma}[Theorem 2.3, \cite{aliagh2022}]\label{s1l4}
Any function  $f(z)=zh(z)\overline{g(z)} \in S^{k*}_{LH}(\alpha)$ with $\omega(0)=0.$ and $n,k\geq 1$ satisfies
\begin{itemize}
\item[(i)] $|a_{kn}|\leq \dfrac{2}{k}(1-\alpha)+\dfrac{1}{kn},$
\item[(ii)] $|b_{kn}|\leq \dfrac{2}{k}(1-\alpha)+\dfrac{2\alpha -1}{kn}.$

The inequalities are sharp for the functions of the form $\overline{\lambda} f_{\alpha}(\lambda z),$ $|\lambda|=1,$ where $f_{\alpha}(z)$ is given by \eqref{s1l3e1}.
\end{itemize}
\end{lemma}
Now we see the relations between $h$, $g$ and $\alpha$ in term of subordination.
\begin{theorem}\label{s2t1}
Let $0\leq \alpha <1.$ Then the function $f(z)=zh(z)\overline{g(z)} \in S^{k*}_{LH}(\alpha)$ if and only if 
\begin{equation*}
\left( z\dfrac{h^{\prime}(z)}{h(z)}-z\dfrac{g^{\prime}(z)}{g(z)}\right) \prec \dfrac{2(1-\alpha)z^{k}}{1-z^{k}}.
\end{equation*}
\end{theorem}
\begin{proof}
Let $f\in S^{k*}_{LH}(\alpha)$ with the form $f(z)=zh(z)\overline{g(z)}.$ Then 
\begin{equation*}
\alpha< Re\left(\dfrac{zf_{z}-\overline{z}f_{\overline{z}}}{f}\right)=Re\left(1+z\dfrac{h^{\prime}(z)}{h(z)}-\overline{z}\overline{\left(\dfrac{g^{\prime}(z)}{g(z)}\right)}\right)=Re\left(1+z\dfrac{h^{\prime}(z)}{h(z)}-z\dfrac{g^{\prime}(z)}{g(z)}\right)
 \end{equation*} 
 if and only if 
 \begin{equation}\label{ss}
1+z\dfrac{h^{\prime}(z)}{h(z)}-z\dfrac{g^{\prime}(z)}{g(z)}=(1-\alpha)p(z^{k})+\alpha=(1-\alpha)\dfrac{1+\psi (z^{k})}{1-\psi (z^{k})}+\alpha
 \end{equation}
for some $p \in \mathcal{P}$ and Schwarz function $\psi$. Therefore, 
\begin{equation*}
z\dfrac{h^{\prime}(z)}{h(z)}-z\dfrac{g^{\prime}(z)}{g(z)}=(1-\alpha)\dfrac{1+\psi ({z^{k}})}{1-\psi (z^{k})}-(1-\alpha)
\end{equation*}
and hence 
\begin{equation*}
\left( z\dfrac{h^{\prime}(z)}{h(z)}-z\dfrac{g^{\prime}(z)}{g(z)}\right) \prec \dfrac{2(1-\alpha)z^{k}}{1-z^{k}}.
\end{equation*}
\end{proof}
The distortion bounds of the functions in the class $S_{LH}^{k*}(\alpha)$ is given by the next theorem.
\begin{theorem}\label{s2t2}
Let $f =zh(z)\overline{g(z)}\in S^{k*}_{LH}(\alpha)$ with $\omega (0)=0$. Then for  $|z|= r<1$, we have
\begin{itemize}
\item[(i)] $\begin{aligned}[t]
 \dfrac{1-(1-2\alpha)r^{k}}{\left(1+r^{k}\right)^{\frac{2(\alpha+k)}{k}}}\exp \left((1-\alpha)\dfrac{-4r^{k}}{k(1+r^{k})}\right) &\leq |f_{z}(z)| \\
&\leq \dfrac{1+(1-2\alpha)r^{k}}{\left(1-r^{k}\right)^{\frac{2(\alpha+k)}{k}}}\exp \left((1-\alpha)\dfrac{4r^{k}}{k(1-r^{k})}\right),
 \end{aligned} $
\item[(ii)] $\begin{aligned}[t]
&\dfrac{r^{k}\left(1-(1-2\alpha)r^{k}\right)}{\left(1+r^{k}\right)^{\frac{2(\alpha+k)}{k}}} \exp \left((1-\alpha)\dfrac{-4r^{k}}{k(1+r^{k})}\right) \leq  |f_{\overline{z}}(z)| \\
 & \hphantom{\dfrac{r^{k}\left(1-(1-2\alpha)r^{k}\right)}{\left(1+r^{k}\right)^{\frac{2(\alpha+k)}{k}}}\exp \left((1-\alpha)\right)}{}
 \leq  \dfrac{r^{k}\left(1+(1-2\alpha)r^{k}\right)}{\left(1-r^{k}\right)^{\frac{2(\alpha+k)}{k}}}\exp \left((1-\alpha)\dfrac{4r^{k}}{k(1-r^{k})}\right),
 \end{aligned} $
\item[(iii)]$ \begin{aligned}[t]
& \dfrac{1-(1-2\alpha)r^k-(1+r^k)^2}{r(1+r^k)^2} \exp \left((1-\alpha)\dfrac{-2r^{k}}{k(1+r^{k})}\right) \leq  |h^{\prime}(z)| \\
& \hphantom{\dfrac{1-(1-2\alpha)r^k-(1+r^k)^2}{r(1+r^k)^2} \exp }{}
   \leq  \dfrac{1+(1-2\alpha)r^k+(1-r^k)^2}{r(1-r^k)^2}\exp \left((1-\alpha)\dfrac{2r^{k}}{k(1-r^{k})}\right),
  \end{aligned}$
\item[(iv)] $ \begin{aligned}[t]
&\dfrac{r^{k-1}(1-(1-2\alpha)r^{k})}{(1+r^{k})^{\frac{2(\alpha+k)-1}{k}}}\exp \left((1-\alpha)\dfrac{(-2r^{k})}{k(1+r^{k})}\right) \leq |g^{\prime}(z)| \\
& \hphantom{\dfrac{r^{k-1}(1-(1-2\alpha)r^{k})}{(1+r^{k})^{\frac{2(\alpha+k)-1}{k}}}\exp \left((1-\alpha)\right)}{}
 \leq \dfrac{r^{k-1}(1+(1-2\alpha)r^{k})}{(1-r^{k})^{\frac{2(\alpha+k)-1}{k}}}\exp \left((1-\alpha)\dfrac{2r^{k}}{k(1-r^{k})}\right).
 \end{aligned} $
\end{itemize}

The inequalities are sharp for functions of the form $\overline{\lambda} f_{\alpha}(\lambda z),$ $|\lambda|=1,$ where $f_{\alpha}(z)$ is given by \eqref{s1l3e1}.
\end{theorem}
To prove the above theorem, we need the following lemma:
\begin{lemma}\label{s2l1}
Let $f(z)=zh(z)\overline{g(z)}\in S^{k*}_{LH}(\alpha)$ is the solution of \eqref{s1e1} with respect to $\omega$ and $\phi(z)=\dfrac{zh(z)}{g(z)}$. Then 
\begin{itemize}
\item[(a)] $(1-\alpha)\dfrac{1-r^{k}}{1+r^{k}}+\alpha \leq \left| z\dfrac{\phi^{\prime}(z)}{\phi (z)}\right| \leq (1-\alpha)\dfrac{1+r^{k}}{1-r^{k}}+\alpha,$  
\item[(b)] $\dfrac{r}{(1+r^{k})^{\frac{2(1-\alpha)}{k}}}\leq|\phi (z)|\leq \dfrac{r}{(1-r^{k})^{\frac{2(1-\alpha)}{k}}},$
\item[(c)] $\dfrac{r^{k}}{1+r^{k}}\leq \left| \dfrac{\omega (z)}{1-\omega (z)} \right| \leq \dfrac{r^{k}}{1-r^{k}}.$
\end{itemize}
\end{lemma}
\begin{proof}
By making use of Lemma \ref{s1l1}, Lemma \ref{s1l2} and expression \eqref{ss}, we have
\begin{equation*}
\dfrac{z \phi ^{\prime}(z)}{\phi (z)}=(1-\alpha) \int _{|\eta|=1}\dfrac{1+\eta z^k}{1-\eta z^k}d\nu (\eta)+ \alpha, \qquad z \in \mathbb{D},
\end{equation*}
for some probability measure $\nu$ on the boundary of $\mathbb{D}$. For $|z|=r <1$, the above expression give
\begin{align*}
\left| \dfrac{z \phi ^{\prime} (z)}{\phi (z)}\right|  \geq \min_{\nu}\left\{ \min_{|z|=r} \left((1-\alpha) \int_{\partial \mathbb{D}} \dfrac{1+\eta z^k}{1-\eta z^k}d\nu (\eta)+ \alpha\right) \right\}  \geq (1-\alpha)\dfrac{1-r^{k}}{1+r^{k}}+\alpha
\end{align*} 
and 
\begin{align*}
\left| \dfrac{z \phi ^{\prime} (z)}{\phi (z)}\right|  \leq (1-\alpha)\int_{\partial \mathbb{D}} \left| \dfrac{1+\eta z^k}{1-\eta z^k} \right| d\nu (\eta)+\alpha \leq (1-\alpha)\dfrac{1+r^{k}}{1-r^{k}}+\alpha.
\end{align*}
This complete the proof of \itshape{(a)}. The proof of $(b)$ directly follows from \itshape{(a)} and using the same argument as in \itshape{(a)}, the expression \eqref{cara2} gives \itshape{(c)}.
\end{proof}

\begin{proof}[Proof of Theorem \ref{s2t2}]
Let $f\in S^{k*}_{LH}(\alpha)$ is of the form \eqref{s1e3} and $\phi (z)=\frac{zh(z)}{g(z)}, z \in \mathbb{D},$ then
\begin{equation}\label{s2t2e1}
f(z)=\phi(z)|g(z)|^{2}
\end{equation}
and 
\begin{equation}\label{s2t2e2}
h(z)=\dfrac{\phi(z)g(z)}{z}
\end{equation}
the function $g(z)$ has the representation of the form (cf. \cite{aliagh2022})
\begin{equation}\label{s2t2e3}
g(z)=\exp \left(\int_{0}^{z}\dfrac{\omega(s)}{1-\omega(s)}\dfrac{\phi^{\prime}(s)}{\phi(s)}ds\right), \qquad z \in \mathbb{D},
\end{equation}
where $\omega$ is defined as in \eqref{s1e1}. Then, \eqref{s2t2e3} gives
\begin{equation*}
|g(z)|^{2}=\exp \left( 2Re\int_{0}^{z}\dfrac{\omega(s)}{1-\omega(s)}\dfrac{\phi^{\prime}(s)}{\phi(s)}ds\right)
\end{equation*}
and the expressions \eqref{s2t2e1} and \eqref{s2t2e2} respectively become
\begin{equation}\label{s2t2e4}
f(z)=\phi(z)\exp \left( 2Re\int_{0}^{z}\dfrac{\omega(s)}{1-\omega(s)}\dfrac{\phi^{\prime}(s)}{\phi(s)}ds\right)
\end{equation}
and 
\begin{equation}\label{s2t2e5}
h(z)=\dfrac{\phi(z)}{z}\exp \left(\int_{0}^{z}\dfrac{\omega(s)}{1-\omega(s)}\dfrac{\phi^{\prime}(s)}{\phi(s)}ds\right)
\end{equation}
The partial derivative of \eqref{s2t2e4} and \eqref{s2t2e5} with respect to $z$, by using the Leibnitz rule, give
\begin{align}\label{s2t2e6}
f_{z}(z)&=\phi^{\prime}(z)\exp\left(2Re\int_{0}^{z}\dfrac{\omega(s)}{1-\omega(s)}\dfrac{\phi^{\prime}(s)}{\phi(s)}ds\right) \nonumber \\
&\hphantom{\exp\left(2Re\int_{0}^{z}\dfrac{\omega(s)}{1-\omega(s)}\dfrac{\phi^{\prime}(s)}{\phi(s)}ds\right)}{}+ \phi(z)\dfrac{\omega(z)}{1-\omega(z)}\dfrac{\phi^{\prime}(z)}{\phi(z)}\exp \left( 2Re\int_{0}^{z}\dfrac{\omega(s)}{1-\omega(s)}\dfrac{\phi^{\prime}(s)}{\phi(s)}ds\right) \nonumber \\
&=\dfrac{\phi^{\prime}(z)}{\phi(z)}f(z)+\dfrac{\phi^{\prime}(z)}{\phi(z)}f(z)\dfrac{\omega(z)}{1-\omega(z)} =f(z)\dfrac{1}{1-\omega(z)}\dfrac{\phi^{\prime}(z)}{\phi(z)}
\end{align}
and
\begin{align}\label{s2t2e7}
h^{\prime}(z)&=\dfrac{z\phi^{\prime}(z)-\phi(z)}{z^{2}}\exp \left(\int_{0}^{z}\dfrac{\omega(s)}{1-\omega(s)}\dfrac{\phi^{\prime}(s)}{\phi(s)}ds\right) \nonumber \\
& \hphantom{\exp \left(\int_{0}^{z}\dfrac{\omega(s)}{1-\omega(s)}\dfrac{\phi^{\prime}(s)}{\phi(s)}ds\right)}{} +\dfrac{\phi(z)}{z}\dfrac{\omega(z)}{1-\omega(z)}\dfrac{\phi^{\prime}(z)}{\phi(z)}\exp \left(\int_{0}^{z}\dfrac{\omega(s)}{1-\omega(s)}\dfrac{\phi^{\prime}(s)}{\phi(s)}ds\right) \nonumber \\
&=\dfrac{z\phi^{\prime}(z)-\phi(z)}{z\phi(z)}h(z)+h(z)\dfrac{\omega(z)}{1-\omega(z)}\dfrac{\phi^{\prime}(z)}{\phi(z)} =h(z)\left(\dfrac{1}{1-\omega(z)}\dfrac{\phi^{\prime}(z)}{\phi(z)}-\dfrac{1}{z}\right)
\end{align}
By making use of the right inequalities of Lemma \ref{s1l3} and Lemma \ref{s2l1} for $|z|=r <1$,  the expressions \eqref{s2t2e6} and \eqref{s2t2e7} become
\begin{align}
|f_{z}(z)|&\leq |f(z)|\left| \dfrac{1}{z(1-\omega(z))}\right| \left| z\dfrac{\phi^{\prime}(z)}{\phi(z)}\right| \nonumber \\
&\leq \dfrac{r\exp\left((1-\alpha)\frac{4r^{k}}{k(1-r^{k})}\right)}{(1-r^{k})^{\frac{2\alpha}{k}}}\dfrac{1}{r(1-r^{k})}\left( (1-\alpha)\dfrac{1+r^{k}}{1-r^{k}}+\alpha\right) \nonumber \\
&= \dfrac{1}{(1-r^{k})^{\frac{2\alpha}{k}+1}}\left( \dfrac{(1-\alpha)(1+r^{k})}{1-r^{k}}+\alpha\right)\exp\left((1-\alpha)\frac{4r^{k}}{k(1-r^{k})}\right) \nonumber \\
&= \dfrac{1+(1-2\alpha)r^{k}}{\left(1-r^{k}\right)^{\frac{2(\alpha+k)}{k}}}\exp \left((1-\alpha)\dfrac{4r^{k}}{k(1-r^{k})}\right) \nonumber
\end{align}
and 
\begin{align}
|h^{\prime}(z)|&\leq |h(z)|\left(\left|\dfrac{1}{1-\omega(z)}\right|\left| \dfrac{\phi^{\prime}(z)}{\phi(z)}\right|+\left| \dfrac{1}{z}\right|\right) \nonumber \\
&\leq \dfrac{\exp\left((1-\alpha)\dfrac{2r^{k}}{k(1-r^{k})}\right)}{(1-r^{k})^{\frac{1}{k}}}\left(\dfrac{1}{r(1-r^{k})}\left( \dfrac{(1-\alpha)(1+r^{k})}{1-r^{k}}+\alpha\right)+\dfrac{1}{r}\right) \nonumber \\
&=\dfrac{1+(1-2\alpha)r^k+(1-r^k)^2}{r(1-r^k)^2}\exp \left((1-\alpha)\dfrac{2r^{k}}{k(1-r^{k})}\right). \nonumber
\end{align}
Similarly, with the help of left inequalities of Lemma \ref{s1l3} and Lemma \ref{s2l1} for $|z|= r <1$,  the expressions \eqref{s2t2e6} and \eqref{s2t2e7} give the left inequalities of {\itshape{(i)}} and {\itshape{(iii)}}. This completes the proof of {\itshape{(i)}} and {\itshape{(iii)}}. The inequalities {\itshape{(ii)}} and {\itshape{(iv)}} follow in the same way as in {\itshape{(i)}} and {\itshape{(iii)}}.

Since the functions $f_{z},~ f_{\overline{z}}, ~h_{z}$ and $g_{z}$ respectively depend on $f,~ f, ~h$ and $g$, the equalities occur in the inequalities of Theorem \ref{s2t2} if the equalities occur in the inequalities of the Lemma \ref{s1l3}.
\end{proof}

\begin{remark}
 Upon substituting $k=1,$ and $\alpha=0$ and $k=1$ Theorem \ref{s2t2} respectively gives distortion bounds for $S_{LH}^{*}(\alpha)$\cite{liupon2018} and $S_{LH}^{*}$\cite{aliabd2016}.
\end{remark}

With the assistance of distortion bounds next, we calculate bounds on the area of the functions in the class $S_{LH}^{k*}(\alpha)$.

\begin{theorem}
The area $Ar$ of the disk $\mathbb{D}_{r}:=\lbrace z \in \mathbb{C}: |z|<r <1 \rbrace$ under the mapping $f \in S_{LH}^{k*}(\alpha)$ satisfy 
\begin{equation}\label{s2t3e1a}
2\pi L_{1} \leq Ar \leq 2\pi L_{2},
\end{equation}
where 
\begin{multline*}
L_{1}=\int_{0}^{r}\left(\dfrac{1-(1-2\alpha)\rho^{k}}{(1+\rho^{k})^{\frac{2(\alpha+k)}{k}}}\right)^{2}\exp\left(\dfrac{-8(1-\alpha)\rho^{k}}{k(1+\rho^{k})}\right)\rho d\rho \\
-\int_{0}^{r}\left(\dfrac{\rho^k(1+(1-2\alpha)\rho^{k})}{(1-\rho^{k})^{\frac{2(\alpha+k)}{k}}}\right)^{2}\exp\left(\dfrac{8(1-\alpha)\rho^{k}}{k(1-\rho^{k})}\right)\rho d\rho
\end{multline*}
and
\begin{multline*}
L_{2}=\int_{0}^{r}\left(\dfrac{1+(1-2\alpha)\rho^{k}}{(1-\rho^{k})^{\frac{2(\alpha+k)}{k}}}\right)^{2}\exp\left(\dfrac{8(1-\alpha)\rho^{k}}{k(1-\rho^{k})}\right)\rho d\rho \\
-\int_{0}^{r}\left(\dfrac{\rho^k(1-(1-2\alpha)\rho^{k})}{(1+\rho^{k})^{\frac{2(\alpha+k)}{k}}}\right)^{2}\exp\left(\dfrac{-8(1-\alpha)\rho^{k}}{k(1+\rho^{k})}\right)\rho d\rho.
\end{multline*}
\end{theorem}

\begin{proof}
Let $f \in S^{k*}_{LH}.$ By making use of \eqref{s1e2}, the area $Ar$ of $\mathbb{D}_{r}$ under $f$ is given by 
\begin{equation}\label{s2t3e2a}
Ar=\int\int_{\mathbb{D}_{r}}J_{f}(z)dxdy=\int\int_{\mathbb{D}_{r}}(|f_{z}|^{2}-|f_{\overline{z}}|^{2})dxdy, \qquad z=x+iy,
\end{equation}
where $J_{f}$ is the Jacobian of the mapping $f.$ Use of Theorem \ref{s2t2} in \eqref{s2t3e2a}, and taking $z=re^{i \theta},~ 0\leq r<1$ and $0\leq \theta \leq 2\pi$ give
\begin{align}\label{s2t3e3a}
\int_{0}^{2\pi}L_{1}d\theta \leq Ar &\leq \int_{0}^{2\pi}L_{2}d\theta \nonumber
\end{align}
This implies \eqref{s2t3e1a}.
\end{proof}

\section{Improved Bohr radius}\label{improved}
For the functions class $ S^{k*}_{LH}(\alpha)$, the Bohr radius is calculated by Alizadeh et al. \cite{aliagh2022}. In this section, we have calculated  improved Bohr radius and  Bohr type inequalities, and illustrated the radii numerically for different values of the parameters $\alpha$ and $k$. The first theorem of this section gives the improved Bohr radius for functions in the family $S_{LH}^{k*}(\alpha)$. The theorem is calculated in support of dilogarithmic function, defined by 
\begin{equation*}
    Li_{2}(z)=\sum_{n=1}^{\infty}\dfrac{z^{n}}{n^2}, \qquad z \in \mathbb{C}.
\end{equation*}
\begin{theorem}\label{s4thm1}
Let $f(z)=zh(z)\overline{g(z)} \in S^{k*}_{LH}(\alpha)$ with $\omega(0)=0.$ Then for any real $\theta,$ the inequality 
\begin{equation*}
|z|\exp \left( \sum_{n=1}^{\infty}|a_{nk}+e^{i\theta}b_{nk}+ka_{nk}b_{nk}||z|^{nk}\right)\leq d(f(0), \partial f(\mathbb{D}))
\end{equation*}
is true for $|z|\leq r_{1},$ where $r_{1}$ is unique root in $(0, 1)$ of 
\begin{equation*}
\dfrac{r2^{\frac{2\alpha}{k}}}{(1-r^{k})^{\frac{2\alpha(3-2\alpha)}{k}}}\exp \left( \left[\dfrac{4(1-\alpha)(2-\alpha)}{1-r^{k}}+(2\alpha-1)Li_{2}(r^{k})+2(1-\alpha)\right]\frac{1}{k}\right)=1.
\end{equation*}
The radius $r_{1}$ is the best possible.
\end{theorem}
\begin{proof}
Let $f(z)=zh(z)\overline{g(z)} \in S^{k*}_{LH}(\alpha)$. In view of Lemma \ref{s1l3}, we have the following relation 
\begin{equation}\label{s2t3e1}
\dfrac{1}{2^{\frac{2\alpha}{k}}e^{\frac{2(1-\alpha)}{k}}}\leq d(f(0), \partial f(\mathbb{D}))\leq 1.
\end{equation}
Then, for any real $\theta$, and through the use of Lemma \ref{s1l4} together with \eqref{s2t3e1}, we have
\begin{align*}
|z|\exp & \left( \sum_{n=1}^{\infty}|a_{nk}+e^{i\theta}b_{nk}+ka_{nk}b_{nk}||z|^{nk}\right) \\
        & \leq r\exp \left( \sum_{n=1}^{\infty}\left( \frac{4}{k}(1-\alpha)+                      \frac{2\alpha}{kn}+\frac{1}{k}\left[2(1-\alpha)+\frac{1}{n}\right]\left[2(1-\alpha)+\frac{2\alpha-1}{n}\right]\right)r^{nk}\right) \\
        &= r \exp \left( \sum_{n=1}^{\infty}\left( \dfrac{4}{k}(1-\alpha)(2-\alpha)+\dfrac{2\alpha}{kn}(3-2\alpha)+\dfrac{2\alpha-1}{kn^{2}}\right)r^{nk}\right) \\
        &=r\exp\left( \dfrac{4(1-\alpha)(2-\alpha)r^{k}}{k(1-r^{k})}-\dfrac{2\alpha (3-2\alpha)}{k}\log (1-r^{k})+\dfrac{2\alpha -1}{k}Li_{2}(r^{k})\right)\\
        & \leq d(f(0), \partial f(\mathbb{D}))
\end{align*}
if and only if
\begin{equation*}
\dfrac{r}{(1-r^{k})^{\frac{2\alpha(3-2\alpha)}{k}}}\exp \left( \dfrac{4(1-\alpha)(2-\alpha)r^{k}}{k(1-r^{k})}+\dfrac{2\alpha-1}{k}Li_{2}(r^{k})\right) \leq \dfrac{1}{2^{\frac{2\alpha}{k}}e^{\frac{2(1-\alpha)}{k}}}.
\end{equation*}
Therefore, for this case, the Bohr radius $r_{1}$ is the unique root in $(0, 1)$ of the equation 
\begin{equation*}
\dfrac{r}{(1-r^{k})^{\frac{2\alpha(3-2\alpha)}{k}}}\exp \left( \dfrac{4(1-\alpha)(2-\alpha)r^{k}}{k(1-r^{k})}+\dfrac{2\alpha-1}{k}Li_{2}(r^{k})\right) = \dfrac{1}{2^{\frac{2\alpha}{k}}e^{\frac{2(1-\alpha)}{k}}},
\end{equation*}
that is, 
\begin{equation*}
\dfrac{r2^{\frac{2\alpha}{k}}}{(1-r)^{\frac{2\alpha(1-\alpha)}{k}}}\exp \left( \left[\dfrac{4(1-\alpha)(2-\alpha)r^{k}}{1-r^{k}}+(2\alpha-1)Li_{2}(r^{k})+2(1-\alpha)\right]\frac{1}{k}\right)=1.
\end{equation*}

To show the uniqueness of $r_{1},$ consider the function $f_{1}:[0, 1)\rightarrow \mathbb{R},$ defined by 
\begin{equation*}
f_{1}(r)=\dfrac{r2^{\frac{2\alpha}{k}}}{(1-r)^{\frac{2\alpha(1-\alpha)}{k}}}\exp \left( \left[\dfrac{4(1-\alpha)(2-\alpha)r^{k}}{1-r^{k}}+(2\alpha-1)Li_{2}(r^{k})+2(1-\alpha)\right]\frac{1}{k}\right)-1.
\end{equation*}
Note that $f_{1}(0)=-1$, $f_{1}(1)=\infty$ and $f_{1}^{\prime}(r)>0,~ \forall r \in (0,1).$ In accordance to intermediate value theorem, $r_{1}$ is the unique root of $f_{1}$.

To show the sharpness of $r_{1}$, consider the function defined in \eqref{s1l3e1} and $r=r_{1}$. For this function 
\begin{equation*}
|a_{kn}|=\dfrac{2}{k}(1-\alpha)+\dfrac{1}{kn}, ~~ |b_{nk}|=\dfrac{2}{k}(1-\alpha)+\dfrac{2\alpha-1}{kn}~~ \text{and} ~~ d(0, \partial f(\mathbb{D}))=\dfrac{1}{2^{\frac{2\alpha}{k}}e^{\frac{2(1-\alpha)}{k}}}
\end{equation*}
Then, 
\begin{align*}
&|z|\exp \left( \sum_{n=1}^{\infty}|a_{nk}+e^{i\theta}b_{nk}+ka_{nk}b_{nk}||z|^{nk}\right) \\
&= r\exp \left( \sum_{n=1}^{\infty}\left( \frac{4}{k}(1-\alpha)+\frac{2\alpha}{kn}+\left[\frac{2}{k}(1-\alpha)+\frac{1}{kn}\right]\left[\frac{2}{k}(1-\alpha)+\frac{2\alpha-1}{kn}\right]\right)r^{nk}\right) \\
&=\dfrac{r_{1}}{(1-r_{1}^{k})^{\frac{2\alpha(3-2\alpha)}{k}}}\exp \left( \dfrac{4(1-\alpha)(2-\alpha)r_{1}^{k}}{k(1-r_{1}^{k})}+\dfrac{2\alpha-1}{k}Li_{2}(r_{1}^{k})\right) =d(f(0), \partial f(\mathbb{D})).
\end{align*}
Therefore, the radius $r_{1}$ is the best possible.
\end{proof}

The next theorem is the sharp Bohr radius of the analytic and co-analytic factors of the functions in the class $S_{LH}^{k*}(\alpha)$ when $|a_{kn}|^2$ and $|b_{kn}|^2$ respectively added in their power series expansion.

\begin{theorem}\label{s4thm2}
Let $f(z)=zh(z)\overline{g(z)} \in S^{k*}_{LH}(\alpha)$ with $\omega(0)=0$, $H(z)=zh(z)$ and $G(z)=zg(z).$ Then for $z\in \mathbb{D}$, 
\begin{itemize}
\item[(i)] the inequality 
\begin{equation*}
|z|\exp \left(\sum_{n=1}^{\infty}\left(|a_{nk}|+k|a_{nk}|^{2}\right)|z|^{nk}\right)\leq d(H(0), \partial H(\mathbb{D}))
\end{equation*}
holds for $|z|=r \leq r_{2}$, where $r_{2}$ is the unique root in $(0, 1)$ of the equation 
\begin{equation*}
\dfrac{2^{\frac{1}{k}}r}{(1-r^{k})^{\frac{5-4\alpha}{k}}}\exp \left(\left(\dfrac{2(1-\alpha)(3-2\alpha)r^{k}}{1-r^{k}}+Li_{2}(r^{k})+1-\alpha \right)\dfrac{1}{k}\right) =1,
\end{equation*}
\item[(ii)] the inequality 
\begin{equation*}
|z|\exp \left(\sum_{n=1}^{\infty}\left(|b_{nk}|+k|b_{nk}|^{2}\right)|z|^{nk}\right)\leq d(G(0), \partial G(\mathbb{D}))
\end{equation*}
holds for $|z|=r \leq r_{3}$, where $r_{3}$ is the unique root in $(0, 1)$ of the equation 
\begin{equation*}
\dfrac{2^{\frac{2\alpha-1}{k}}r}{(1-r^{k})^{\frac{(2\alpha-1)(5-4\alpha)}{k}}}\exp \left(\left(\dfrac{2(1-\alpha)(3-2\alpha)r^{k}}{1-r^{k}}+(2\alpha-1)^{2}Li_{2}(r^{k})+1-\alpha \right)\dfrac{1}{k}\right) =1.
\end{equation*}
\end{itemize}
\end{theorem}
\begin{proof}
 Let $f(z)=zh(z)\overline{g(z)}\in S^{k*}_{LH}(\alpha)$. From Lemma \ref{s1l3}, we have the relations
\begin{equation}\label{s2t4e1}
\dfrac{1}{2^{\frac{1}{k}}e^{\frac{1-\alpha}{k}}}\leq d(H(0), \partial H(\mathbb{D}))\leq 1
\end{equation}
and
\begin{equation}\label{s2t4e2}
\dfrac{1}{2^{\frac{2\alpha-1}{k}}e^{\frac{1-\alpha}{k}}}\leq d(G(0), \partial G(\mathbb{D}))\leq 1
\end{equation}
(i) By application of Lemma \ref{s1l4} and \eqref{s2t4e1}, and  for $|z|=r<1$ we have 
\begin{align*}
|z|\exp & \left(\sum_{n=1}^{\infty}\left(|a_{nk}|+k|a_{nk}|^{2}\right)|z|^{nk}\right) \\
%& \leq r \exp \left(\sum_{n=1}^{\infty}\left( \dfrac{2}{k}(1-\alpha)+\dfrac{1}{kn}+k\left(\dfrac{2}{k}(1-\alpha)+\dfrac{1}{kn}\right)^{2} \right)r^{kn} \right)\\
&=r\exp\left( \sum_{n=1}^{\infty} \left( \dfrac{2(1-\alpha)(3-2\alpha)}{k}+\dfrac{5-4\alpha}{k}\dfrac{1}{n}+\dfrac{1}{k}\dfrac{1}{n^{2}}\right)r^{nk}\right) \\
&=r \exp \left( \dfrac{2(1-\alpha)(3-2\alpha)r^{k}}{k(1-r^{k})}-\dfrac{5-4\alpha}{k}\log(1-r^{k})+\dfrac{1}{k}Li_{2}(r^{k})\right) \\
& \leq d(H(0), \partial H(\mathbb{D}))
\end{align*} 
if and only if 
\begin{equation*}
\dfrac{r}{(1-r^{k})^{\frac{5-4\alpha}{k}}}\exp \left(\left(\dfrac{2(1-\alpha)(3-2\alpha)r^{k}}{1-r^{k}}+Li_{2}(r^{k}) \right)\dfrac{1}{k}\right)\leq \dfrac{1}{2^{\frac{1}{k}}e^{\frac{1-\alpha}{k}}},
\end{equation*}
that is,
\begin{equation*}
\dfrac{2^{\frac{1}{k}}r}{(1-r^{k})^{\frac{5-4\alpha}{k}}}\exp \left(\left(\dfrac{2(1-\alpha)(3-2\alpha)r^{k}}{1-r^{k}}+Li_{2}(r^{k})+1-\alpha \right)\dfrac{1}{k}\right)\leq 1.
\end{equation*}
Therefore, for this case, the Bohr radius $r_{2}$ is the unique root of the equation
\begin{equation*}
\dfrac{2^{\frac{1}{k}}r}{(1-r^{k})^{\frac{5-4\alpha}{k}}}\exp \left(\left(\dfrac{2(1-\alpha)(3-2\alpha)r^{k}}{1-r^{k}}+Li_{2}(r^{k})+1-\alpha \right)\dfrac{1}{k}\right) =1.
\end{equation*}
To show the uniqueness of $r_{2},$ consider the function $f_{2}:[0, 1)\rightarrow \mathbb{R},$ defined by 
\begin{equation*}
f_{2}(r)=\dfrac{2^{\frac{1}{k}}r}{(1-r^{k})^{\frac{5-4\alpha}{k}}}\exp \left(\left(\dfrac{2(1-\alpha)(3-2\alpha)r^{k}}{1-r^{k}}+Li_{2}(r^{k})+1-\alpha \right)\dfrac{1}{k}\right) -1.
\end{equation*}
Note that $f_{2}(0)=-1$, $\lim_{r\rightarrow 1}f_{2}(r)=\infty$ and $f_{2}^{\prime}(r)>0,~ \forall r \in (0,1).$ In accordance to intermediate value theorem, $r_{2}$ is the unique root of $f_{2}$.
To show the sharpness, consider the function \eqref{h} and $r=r_{2}$. Then $$ |a_{nk}|=\dfrac{2}{k}(1-\alpha)+\dfrac{1}{kn}~~ \text{and} ~~ d(H(0), \partial H(\mathbb{D}))=\dfrac{1}{2^{\frac{1}{k}}e^{\frac{1-\alpha}{k}}}.$$
Therefore,
\begin{multline*}
|z|\exp \left(\sum_{n=1}^{\infty}\left(|a_{nk}|+k|a_{nk}|^{2}\right)|z|^{nk}\right)= r_{2}+\exp \left( \dfrac{2}{k}(1-\alpha)+\dfrac{1}{kn}+k\left(
\dfrac{2}{k}(1-\alpha)+\dfrac{1}{kn}\right)^{2} \right)r_{2}^{kn} \\
= \dfrac{r_{2}}{(1-r_{2}^{k})^{\frac{5-4\alpha}{k}}}\exp \left( \dfrac{2(1-\alpha)(3-2\alpha)r_{2}^{k}}{k(1-r_{2}^{k})}+\dfrac{1}{k}Li_{2}(r_{2}^{k})\right) = d(H(0), \partial H(\mathbb{D}))
\end{multline*}
Therefore, $r_{2}$ is the best possible.

(ii) For $|z|=r<1$, the use of Lemma \ref{s1l4} along with \eqref{s2t4e2} give 
\begin{align*}
&|z|\exp \left(\sum_{n=1}^{\infty}\left(|b_{nk}|+k|b_{nk}|^{2}\right)|z|^{nk}\right) \\
& \leq r \exp \left( \dfrac{2}{k}(1-\alpha)+\dfrac{2\alpha-1}{kn}+k\left( \dfrac{2}{k}(1-\alpha)+\dfrac{2\alpha-1}{kn}\right)^{2} \right)r^{kn} \\
&=r\exp\left( \sum_{n=1}^{\infty} \left( \dfrac{2(1-\alpha)(3-2\alpha)}{k}+\dfrac{(2\alpha-1)(5-4\alpha)}{k}\dfrac{1}{n}+\dfrac{(2\alpha-1)^{2}}{k}\dfrac{1}{n^{2}}\right)r^{nk}\right) \\
&=r \exp \left( \dfrac{2(1-\alpha)(3-2\alpha)r^{k}}{k(1-r^{k})}-\dfrac{(2\alpha-1)(5-4\alpha)}{k}\log(1-r^{k})+\dfrac{(2\alpha-1)^{2}}{k}Li_{2}(r^{k})\right) \\
& \leq d(G(0), \partial G(\mathbb{D}))
\end{align*} 
if and only if 
\begin{equation*}
\dfrac{r}{(1-r^{k})^{\frac{(2\alpha-1)(5-4\alpha)}{k}}}\exp \left( \dfrac{2(1-\alpha)(3-2\alpha)r^{k}}{k(1-r^{k})}+\dfrac{(2\alpha-1)^{2}}{k}Li_{2}(r^{k})\right)\leq \dfrac{1}{2^{\frac{2\alpha-1}{k}}e^{\frac{1-\alpha}{k}}},
\end{equation*}
that is,
\begin{equation*}
\dfrac{2^{\frac{2\alpha-1}{k}}r}{(1-r^{k})^{\frac{(2\alpha-1)(5-4\alpha)}{k}}}\exp \left(\left(\dfrac{2(1-\alpha)(3-2\alpha)r^{k}}{1-r^{k}}+(2\alpha-1)^{2}Li_{2}(r^{k})+1-\alpha \right)\dfrac{1}{k}\right) \leq1.
\end{equation*}
Therefore, for this case, the Bohr radius $r_{3}$ is the unique root of the equation
\begin{equation*}
\dfrac{2^{\frac{2\alpha-1}{k}}r}{(1-r^{k})^{\frac{(2\alpha-1)(5-4\alpha)}{k}}}\exp \left(\left(\dfrac{2(1-\alpha)(3-2\alpha)r^{k}}{1-r^{k}}+(2\alpha-1)^{2}Li_{2}(r^{k})+1-\alpha \right)\dfrac{1}{k}\right)=1.
\end{equation*}
To show the uniqueness of $r_{3},$ consider the function $f_{3}:[0, 1)\rightarrow \mathbb{R},$ defined by 
\begin{equation*}
f_{3}(r)=\dfrac{2^{\frac{2\alpha-1}{k}}r}{(1-r^{k})^{\frac{(2\alpha-1)(5-4\alpha)}{k}}}\exp \left(\left(\dfrac{2(1-\alpha)(3-2\alpha)r^{k}}{1-r^{k}}+(2\alpha-1)^{2}Li_{2}(r^{k})+1-\alpha \right)\dfrac{1}{k}\right)-1.
\end{equation*}
Note that $f_{3}(0)=-1$, $\lim_{r\rightarrow 1}f_{3}(r)=\infty$ and $f_{3}^{\prime}(r)>0,~ \forall r \in (0,1).$ In accordance to intermediate value theorem, $r_{3}$ is the unique root of $f_{3}$.
To show the sharpness, consider the function \eqref{g} and $r=r_{3}$. For this function $$ |b_{nk}|=\dfrac{2}{k}(1-\alpha)+\dfrac{2\alpha-1}{kn}~~ \text{and} ~~ d(G(0), \partial G(\mathbb{D}))=\dfrac{1}{2^{\frac{2\alpha-1}{k}}e^{\frac{1-\alpha}{k}}}.$$
Therefore,
\begin{align*}
&|z|\exp \left(\sum_{n=1}^{\infty}\left(|b_{nk}|+k|b_{nk}|^{2}\right)|z|^{nk}\right) \\
&= r_{3} \exp \left( \dfrac{2}{k}(1-\alpha)+\dfrac{2\alpha-1}{kn}+k\left( \dfrac{2}{k}(1-\alpha)+\dfrac{2\alpha-1}{kn}\right)^{2} \right)r_{3}^{kn} \\
&= \dfrac{r_{3}}{(1-r_{3}^{k})^{\frac{(2\alpha-1)(5-4\alpha)}{k}}}\exp \left( \dfrac{2(1-\alpha)(3-2\alpha)r_{3}^{k}}{k(1-r_{3}^{k})}+\dfrac{(2\alpha-1)^{2}}{k}Li_{2}(r_{3}^{k})\right) = d(G(0), \partial G(\mathbb{D}))
\end{align*}
Therefore, $r_{3}$ is the best possible.
\end{proof}

We calculate the improved sharp Bohr radius for the function in the class $S_{LH}^{k*}(\alpha)$ by adding the modulus of the sum of its analytic and co-analytic factor to its series expansion.

\begin{theorem}\label{s4thm3}
Let $f(z)=zh(z)\overline{g(z)} \in S^{k*}_{LH}(\alpha)$ with $\omega(0)=0$, $|h(z)|\leq1$ and $|g(z)|\leq1$. Then for any real $\theta$, we have 
\begin{equation*}
|z|\exp \left( |h(z)+g(z)|+\sum_{n=1}^{\infty}|a_{nk}+e^{i\theta}b_{nk}||z|^{nk}\right) \leq d(f(0), \partial f(\mathbb{D}))
\end{equation*}
for $|z|=r\leq r_{4},$ where $r_{4}$ is the unique root in $(0, 1)$ of the equation
\begin{equation*}
\left(\dfrac{2}{1-r^{k}}\right)^{\frac{2\alpha}{k}}e^{2}r\exp \left(\left( \dfrac{2(1-\alpha)r^{k}}{1-r^{k}}+1-\alpha \right)\dfrac{2}{k}\right)=1
\end{equation*}
The radius $r_{4}$ is the best possible.
\end{theorem}
\begin{proof}
let $f(z)=zh(z)\overline{g(z)}\in S^{k*}_{LH}(\alpha),$ then from Lemma \ref{s1l3} and \ref{s1l4} with $|z|=r<1$, we have
\begin{align*}
|z|\exp & \left( |h(z)+g(z)|+\sum_{n=1}^{\infty}|a_{nk}+e^{i\theta}b_{nk}||z|^{nk}\right) \\
& \leq e^{2}r \exp \left( \sum_{n=1}^{\infty} \left( \dfrac{4}{k}(1-\alpha)+\dfrac{2\alpha}{kn}\right)r^{nk}\right) \leq d(f(0), \partial f(\mathbb{D}))
\end{align*}
if and only if 
\begin{equation*}
\dfrac{e^{2}r}{(1-r^{k})^{\frac{2\alpha}{k}}}\exp \left( \dfrac{4(1-\alpha)r^{k}}{k(1-r^{k})}\right) \leq \dfrac{1}{2^{\frac{2\alpha}{k}}e^{\frac{2(1-\alpha)}{k}}},
\end{equation*} 
That is,
\begin{equation*}
\left(\dfrac{2}{1-r^{k}}\right)^{\frac{2\alpha}{k}}e^{2}r\exp \left( \dfrac{2(1-\alpha)}{k}\dfrac{1+r^{k}}{1-r^{k}}\right)\leq1
\end{equation*}
Here, the Bohr radius is $r_4$ and is the solution of 
\begin{equation*}
\left(\dfrac{2}{1-r^{k}}\right)^{\frac{2\alpha}{k}}e^{2}r\exp \left( \dfrac{2(1-\alpha)}{k}\dfrac{1+r^{k}}{1-r^{k}}\right)=1.
\end{equation*} 
The function $f_{4}:[0, 1)\rightarrow \mathbb{R},$ defined by 
\begin{equation*}
f_{4}(r)=\left(\dfrac{2}{1-r^{k}}\right)^{\frac{2\alpha}{k}}e^{2}r\exp \left( \dfrac{2(1-\alpha)}{k}\dfrac{1+r^{k}}{1-r^{k}}\right)-1.
\end{equation*}
gives the uniqueness of $r_4$. Note that $f_{4}(0)=-1$, $\lim_{r\rightarrow 1}f_{4}(r)=\infty$ and $f_{4}^{\prime}(r)>0,~ \forall r \in (0,1).$ In accordance with the intermediate value theorem, $r_{4}$ is the unique root of $f_{4}$.
To show the sharpness of $r_{4}$, consider the function of the form \eqref{s1l3e1} and $r=r_{4}$
\begin{align*}
&|z|\exp \left( |h(z)+g(z)|+\sum_{n=1}^{\infty}|a_{nk}+e^{i\theta}b_{nk}||z|^{nk}\right) \\
&= \dfrac{e^{2}r_{4}}{(1-r_{4}^{k})^{\frac{2\alpha}{k}}}\exp \left( \dfrac{4(1-\alpha)r_{4}^{k}}{k(1-r_{4}^{k})}\right) = d(f(0), \partial f(\mathbb{D}))
\end{align*}
Therefore, $r_{4}$ is the best possible.
\end{proof}

\begin{theorem}\label{s4thm4}
For any function $f(z)=zh(z)\overline{g(z)}$ in the class $S^{k*}_{LH}(\alpha)$ in $\mathbb{D}$ with $|h(z)|\leq 1,$ $|g(z)|\leq 1$ and for any real $\theta$ satisfy the inequality 
\begin{equation*}
|f(z)|+|z|\exp \left( \sum_{n=1}^{\infty}|a_{nk}+e^{i\theta} b_{nk}||z|^{nk}\right) \leq d(f(0), \partial f(\mathbb{D}))
\end{equation*}
for $|z|=r \leq r_{5}$, where $r_{5}$ is the unique root, in $(0, 1),$ of the equation equation
\begin{equation*}
r+\dfrac{r}{(1-r^{k})^{\frac{2\alpha}{k}}}\exp \left( \dfrac{4(1-\alpha)r^{k}}{k(1-r^{k})}\right)=\dfrac{1}{2^{\frac{2\alpha}{k}}e^{\frac{2(1-\alpha)}{k}}}
\end{equation*}
and $r_{5}$ is the best possible.
\end{theorem}
\begin{proof}
Let $f(z)=zh(z)\overline{g(z)}\in S^{k*}_{LH}(\alpha).$ In view of Lemma \ref{s1l4}, we have 
\begin{align*}
|f(z)| & +|z|\exp \left( \sum_{n=1}^{\infty}|a_{nk}+e^{i\theta} b_{nk}||z|^{nk}\right) \\
&\leq r+r\exp \left( \dfrac{4(1-\alpha)r^{k}}{k}\dfrac{1}{1-r^{k}}-\dfrac{2\alpha}{k}\log(1-r^{k})\right)  \leq d(f(0), \partial f(\mathbb{D}))
\end{align*}
if and only if 
\begin{equation*}
r+\dfrac{r}{(1-r^{k})^{\frac{2\alpha}{k}}}\exp \left( \dfrac{4(1-\alpha)r^{k}}{k(1-r^{k})}\right) \leq \dfrac{1}{2^{\frac{2\alpha}{k}}e^{\frac{2(1-\alpha)}{k}}}.
\end{equation*}
In this case, the solution of 
\begin{equation*}
r+\dfrac{r}{(1-r^{k})^{\frac{2\alpha}{k}}}\exp \left( \dfrac{4(1-\alpha)r^{k}}{k(1-r^{k})}\right) = \dfrac{1}{2^{\frac{2\alpha}{k}}e^{\frac{2(1-\alpha)}{k}}}.
\end{equation*}
gives the Bohr radius denoted by $r_5$. Consider the function $f_{5}:[0, 1)\rightarrow \mathbb{R},$ defined by 
\begin{equation*}
f_{5}(r)=2^{\frac{2\alpha}{k}}e^{\frac{2(1-\alpha)}{k}}r+\dfrac{2^{\frac{2\alpha}{k}}r}{(1-r^{k})^{\frac{2\alpha}{k}}}\exp \left( \dfrac{2(1-\alpha)(1+r^{k})}{k(1-r^{k})}\right) - 1.
\end{equation*}
It is observed that $f_{5}(0)=-1$, $\lim_{r\rightarrow 1}f_{5}(r)=\infty$ and $f_{5}^{\prime}(r)>0,~ \forall r \in (0,1).$ In accordance to intermediate value theorem, $r_{5}$ is the unique root of $f_{5}$.
The function defined in \eqref{s1l3e1} will give the sharpness of $r_{5}.$
\end{proof}

The next result gives the Bohr type inequalities for the functions in the class $S_{LH}^{k*}(\alpha)$ when the modulus of $f_{z}$ is added  to its power series expansion.

\begin{theorem}\label{s4thm5}
Every function $f(z)=zh(z)\overline{g(z)}$ in the class $S^{k*}_{LH}(\alpha)$ with $\omega (0)=0$ and for any real $\theta$ satisfy the inequality 
\begin{equation*}
|zf_{z}(z)|+|z|\exp \left(\sum_{n=1}^{\infty}|a_{nk}+e^{i\theta}b_{nk}||z|^{nk}\right) \leq d(f(0),\partial f(\mathbb{D}))
\end{equation*}
for $|z|=r_{6}$, where $r_{6}$ is the unique root of the equation 
\begin{equation*}
\dfrac{2^{\frac{2\alpha}{k}}r(2-(1+2\alpha)r^{k}+r^{2k})}{(1-r^{k})^{\frac{2(\alpha+k)}{k}}} \exp \left( \dfrac{2(1-\alpha)(r^{k}+1)}{k(1-r^{k})}\right)=1.
\end{equation*}
Here, $r_{6}$ is the best possible.
\end{theorem}
\begin{proof}
let $f(z)=zh(z)\overline{g(z)}\in S^{k*}_{LH}(\alpha),$ then from Lemma \ref{s1l3} and Theorem \ref{s2t2} with $|z|=r$ and for any real $\theta$, we have
\begin{align*}
|zf_{z}(z) & |+\exp \left(\sum_{n=1}^{\infty}|a_{nk}|e^{i\theta}b_{nk}||z|^{nk}\right) \\
&\leq \dfrac{r(1+(1-2\alpha)r^{k})}{(1-r^{k})^{\frac{2(\alpha+k)}{k}}}\exp\left(\dfrac{4(1-\alpha)r^{k}}{k(1-r^k)}\right)+\dfrac{r}{(1-r^k)^{\frac{2\alpha}{k}}} \exp\left(\dfrac{4(1-\alpha)r^{k}}{k(1-r^k)}\right)\\
&\leq d(f(0), \partial f(\mathbb{D}))
\end{align*}
if and only if 
\begin{equation*}
\dfrac{r(2-(1+2\alpha)r^k+r^{2k})}{(1-r^k)^{\frac{2(\alpha +k)}{k}}}\exp\left( \dfrac{4(1-\alpha)r^k}{k(1-r^k)}\right)  \leq \dfrac{1}{2^{\frac{2\alpha}{k}}e^{\frac{2(1-\alpha)}{k}}}
\end{equation*} 
That is,
\begin{equation*}
\dfrac{2^{\frac{2\alpha}{k}}r(2-(1+2\alpha)r^{k}+r^{2k})}{(1-r^{k})^{\frac{2(\alpha+k)}{k}}} \exp \left( \dfrac{2(1-\alpha)(r^{k}+1)}{k(1-r^{k})}\right) \leq 1.
\end{equation*}
For this case $r_{6}$ is the Bohr radius, where $r_{6}$ is the unique root in $(0, 1)$ of the equation
\begin{equation*}
\dfrac{2^{\frac{2\alpha}{k}}r(2-(1+2\alpha)r^{k}+r^{2k})}{(1-r^{k})^{\frac{2(\alpha+k)}{k}}} \exp \left( \dfrac{2(1-\alpha)(r^{k}+1)}{k(1-r^{k})}\right)=1.
\end{equation*} 
To show the uniqueness of $r_{6},$ consider the function $f_{6}:[0, 1)\rightarrow \mathbb{R},$ defined by 
\begin{equation*}
f_{6}(r)=\dfrac{2^{\frac{2\alpha}{k}}r(2-(1+2\alpha)r^{k}+r^{2k})}{(1-r^{k})^{\frac{2(\alpha+k)}{k}}} \exp \left( \dfrac{2(1-\alpha)(r^{k}+1)}{k(1-r^{k})}\right)-1.
\end{equation*}
It is clear that $f_{6}(0)=-1$, $\lim_{r\rightarrow 1}f_{6}(r)=\infty$  $f_{6}^{\prime}(r)>0,~ \forall r \in (0,1)$, and hence $r_6$ is unique root of $f_6$ according to intermediate value theorem.
To show the sharpness of $r_{6}$, consider the function of the form \eqref{s1l3e1} and $r=r_{6}$
\begin{align*}
&|zf_{z}(z)|+|z|\exp \left(\sum_{n=1}^{\infty}|a_{nk}+e^{i\theta}b_{nk}||z|^{nk}\right) \\
&= \dfrac{r(2-(1+2\alpha)r^k+r^{2k})}{(1-r^k)^{\frac{2(\alpha +k)}{2}}}\exp\left( \dfrac{4(1-\alpha)r^k}{k(1-r^k)}\right)= d(f(0), \partial f(\mathbb{D}))
\end{align*}
Therefore, $r_{6}$ is the best possible.
\end{proof}

\subsection{Numerical and graphical illustration}

In this subsection, the variations of Bohr radius and improved Bohr radius with respect to the parameters $\alpha$ and $k$ are shown numerically as well as graphically. The improved Bohr radius for the functions $f$, $h$ and $g$ have been computed and compared with the Bohr radius obtained in \cite[Theorem 3.1]{aliagh2022}, where $f(z)=zh(z)\overline{g(z)}\in S_{LH}^{k*}(\alpha).$ In Table \ref{tab1}, corresponding to each $\alpha,$ the numerical results presented in the second row represents the Bohr radius obtained in \cite[Theorem 3.1, (i)]{aliagh2022} whereas the first row represents improved Bohr radius obtained in Theorem \ref{s4thm1} for different values of $k$. The graphical representation of Table \ref{tab1} is shown in Figure \ref{fig1} for $(\alpha,k)=(0,1),$ $(0.2,1),$ $(0.4,1),$ $(0.6,1),$ $(0,2),$ $(0,3),$ and $(0,4)$, where the dashed and solid curves, respectively, the graphs of the functions whose roots are Bohr and improved Bohr radius. In a similar manner, Bohr radius in \cite[Theorem 3.1 (ii) and (ii)]{aliagh2022} and improved Bohr radius obtained in Theorem \ref{s4thm2} for $h$ and $g$ are given in Table \ref{tab2} and Table \ref{tab3}  respectively. The corresponding graphical illustrations are presented in Figure \ref{fig2} and Figure \ref{fig3}. The numerical and graphical evidence of the improved Bohr inequality given in Theorem \ref{s4thm3}, Theorem \ref{s4thm4} and Theorem \ref{s4thm5}, respectively, are in Table \ref{tab4} and Figure \ref{fig4}, Table \ref{tab5} and Figure \ref{fig5}, and Table \ref{tab6} and Figure \ref{fig6}. All these calculations and figures are carried out using Matlab $R2013a$.

\begin{sidewaystable}
\begin{minipage}[h]{4.3in}
\fontsize{0.3cm}{0cm}
\caption{Comparison of improved Bohr radius $r_{1}$ (obtained in Theorem \ref{s4thm1}) with Bohr radius (see $(i)$, Theorem 3.1, \cite{aliagh2022})}
\label{tab1}
\begin{tabular}{|c|ccccccc|c}
\hline
\backslashbox{$\alpha \downarrow$}{$k \rightarrow$}& $1$ & $2$ & $3$ & $4$ & $7$ & $10$ &\\
\hline
$0$   & $0.0758$ & $0.2754$ & $0.4234$ & $0.5248$ & $0.6918$ & $0.7727$ &\\
      & $0.0908$ & $0.3013$ & $0.4494$ & $0.5489$ & $0.7089$ & $0.7867$ &\\
 \hline
$0.2$ & $0.0856$ & $0.2927$ & $0.4408$ & $0.5410$ & $0.7039$ & $0.7821$ &\\
      & $0.1019$ & $0.3192$ & $0.4671$ & $0.5650$ & $0.7216$ & $0.7958$ &\\
 \hline
$0.4$ & $0.0974$ & $0.3120$ & $0.4601$ & $0.5586$ & $0.7170$ & $0.7922$ &\\
      & $0.1149$ & $0.3390$ & $0.4862$ & $0.5822$ & $0.7341$ & $0.8054$ &\\
 \hline
$0.6$ & $0.1117$ & $0.3342$ & $0.4816$ & $0.5781$ & $0.7311$ & $0.8031$ &\\
      & $0.1302$ & $0.3608$ & $0.5068$ & $0.6007$ & $0.7473$ & $0.8155$ &\\
 %\hline
%0.8 & 0.1169 & 0.3420 & 0.4890 & 0.5848 & 0.7360 & 0.8068 &\\
  %  & 0.1486 & 0.3855 & 0.5297 & 0.6209 & 0.7616 & 0.8256 &\\
 \hline
$0.99$& $0.1516$ & $0.3894$ & $0.5332$ & $0.6240$ & $0.7638$ & $0.8281$ &\\
      & $0.1702$ & $0.4126$ & $0.5543$ & $0.6424$ & $0.7766$ & $0.8378$ &\\
\hline
\end{tabular}
\end{minipage}
\fontsize{0.3cm}{0cm}
\hspace{1cm}
\begin{minipage}[h]{4.3in}
\caption{Comparison of improved Bohr radius $r_{2}$ (obtained in $(ii)$ Theorem \ref{s4thm2}) with Bohr radius (see $(ii)$, Theorem 3.1, \cite{aliagh2022})}
\label{tab2}
\begin{tabular}{|c|ccccccc|c}
\hline
\backslashbox{$\alpha \downarrow$}{$k \rightarrow$}& $1$& 2&3&4&7&10&\\
\hline
$0$   & $0.0729$ & $0.2700$ & $0.4178$ & $0.5197$ & $0.6880$ & $0.7679$ &\\
      & $0.1222$ & $0.3496$ & $0.4963$ & $0.5913$ & $0.7406$ & $0.8104$ &\\
 \hline
$0.2$ & $0.0906$ & $0.3011$ & $0.4493$ & $0.5488$ & $0.7097$ & $0.7866$ &\\
      & $0.1460$ & $0.3820$ & $0.5265$ & $0.6181$ & $0.7596$ & $0.8249$ &\\
 \hline
$0.4$ & $0.1146$ & $0.3385$ & $0.8457$ & $0.5818$ & $0.7338$ & $0.8052$ &\\
      & $0.1753$ & $0.4187$ & $0.5597$ & $0.6471$ & $0.7798$ & $0.8402$ &\\
 \hline
$0.6$ & $0.1478$ & $0.3844$ & $0.5287$ & $0.6200$ & $0.7609$ & $0.8260$ &\\
      & $0.2126$ & $0.4611$ & $0.5969$ & $0.6790$ & $0.8016$ & $0.8566$ &\\
% \hline
%0.8&0.1960&0.4428&0.5810&0.6654&0.7924&0.8496&\\
 %  &0.2621&0.5114&0.6399&0.7155&0.8259&0.8747&\\
 \hline
$0.99$&$0.2669$ & $0.5166$ & $0.6439$ & $0.7188$ & $0.8280$ & $0.8762$ &\\
      &$0.3289$ & $0.5735$ & $0.6903$ & $0.7573$ & $0.8531$ & $0.8948$ &\\
\hline
\end{tabular}
\end{minipage}
\\[0.5cm]
\begin{minipage}[h]{4.2in}
\captionof{figure}{Graph of $f_{1}$ (see Theorem \ref{s4thm1}) and the functions whose roots are the Bohr radius (page 57 of \cite{aliagh2022}).}
\vspace{-0.4cm}
\includegraphics[width=11cm, height=7cm]{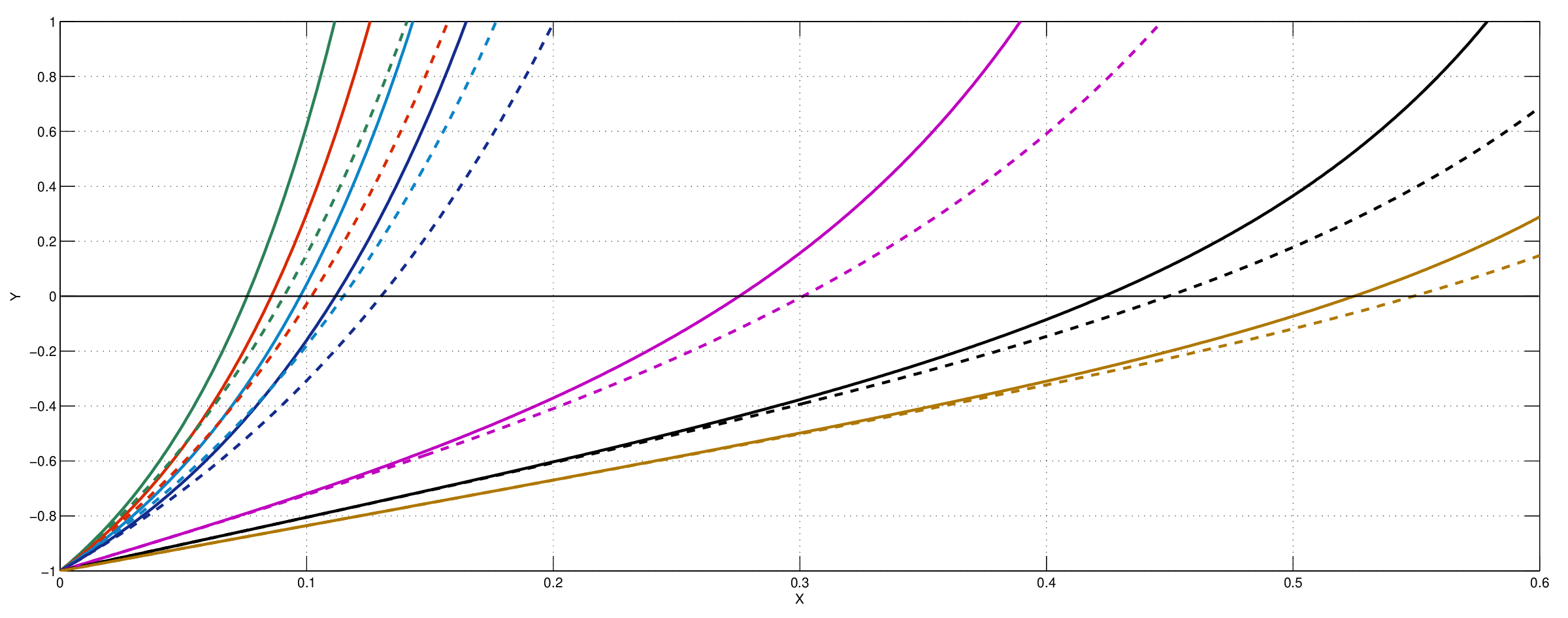}
\label{fig1}
\end{minipage}
\hspace{1cm}
\begin{minipage}[h]{4.2in}
\captionof{figure}{Graph of $f_{2}$ (see Theorem \ref{s4thm2}) and the functions whose roots are the Bohr radius (see page 57 of \cite{aliagh2022}) for $h.$}
\vspace{-0.4cm}
\includegraphics[width=11cm, height=7cm]{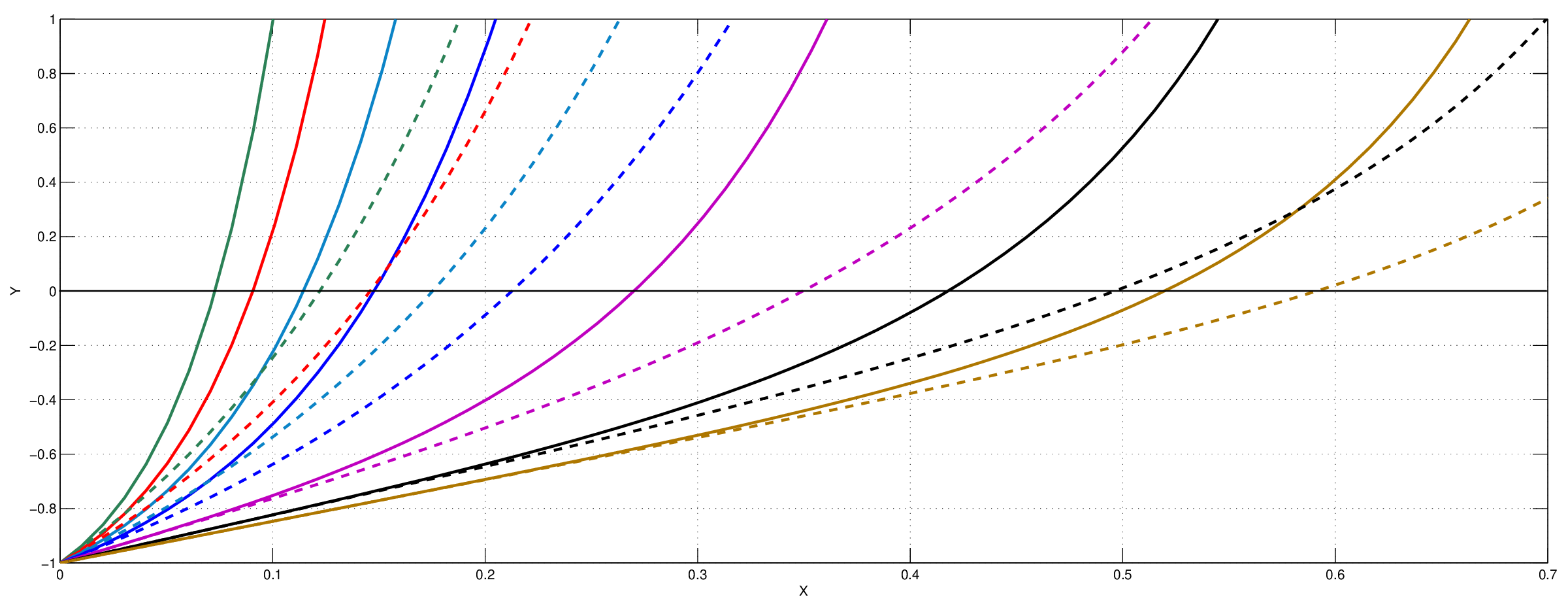}
\label{fig2}
\end{minipage}
\end{sidewaystable}

\begin{sidewaystable}
\begin{minipage}[h]{4.3in}
\fontsize{0.3cm}{0cm}
\caption{Comparison of improved Bohr radius $r_{3}$ (obtained in $(ii),$ Theorem \ref{s4thm2}) with Bohr radius (see $(iii)$, Theorem 3.1, \cite{aliagh2022})}
\label{tab3}
\begin{tabular}{|c|ccccccc|c}
\hline
\backslashbox{$\alpha \downarrow$}{$k \rightarrow$}& $1$& 2&3&4&7&10&\\
\hline
$0$   & $0.2771$ & $0.5264$ & $0.6519$ & $0.7255$ & $0.8324$ & $0.8796$ &\\
      & $0.3659$ & $0.6049$ & $0.7152$ & $0.7778$ & $0.8662$ & $0.9044$ &\\
 \hline
$0.2$ & $0.2788$ & $0.5280$ & $0.6532$ & $0.7267$ & $0.8332$ & $0.8801$ &\\
      & $0.3609$ & $0.6008$ & $0.7120$ & $0.7751$ & $0.8651$ & $0.9031$ &\\
 \hline
$0.4$ & $0.2795$ & $0.5287$ & $0.6538$ & $0.7271$ & $0.8336$ & $0.8803$ &\\
      & $0.3552$ & $0.5960$ & $0.7083$ & $0.7720$ & $0.8625$ & $0.9016$ &\\
 \hline
$0.6$ & $0.2788$ & $0.5280$ & $0.6532$ & $0.7266$ & $0.8332$ & $0.8800$ &\\
      & $0.3489$ & $0.5906$ & $0.7040$ & $0.7685$ & $0.8603$ & $0.9000$ &\\
 \hline
%0.8&0.2050&0.4527&0.5896&0.7629&0.7974&0.8534&\\
 %  &0.3416&0.5844&0.6990&0.7645&0.8577&0.8981&\\
% \hline
$0.99$& $0.2720$ & $0.5215$ & $0.6479$ & $0.7221$ & $0.8302$ & $0.8779$ &\\
      & $0.3338$ & $0.5777$ & $0.6937$ & $0.7601$ & $0.8549$ & $0.8961$ &\\
\hline
\end{tabular}
\end{minipage}
\hspace{1cm}
\begin{minipage}[h]{4.3in}
\fontsize{0.3cm}{0cm}
\caption{Improved Bohr radius $r_{4}$ obtained in Theorem \ref{s4thm3}}
\label{tab4}
\begin{tabular}{|c|ccccccc|c}
\hline
\backslashbox{$\alpha \downarrow$}{$k \rightarrow$}& $1$& 2&3&4&7&10&\\
\hline
$0$    & $0.0170$ & $0.0459$ & $0.0694$ & $0.0820$ & $0.1017$ & $0.1108$ &\\
$0.2$  & $0.0192$ & $0.0527$ & $0.0724$ & $0.0846$ & $0.1035$ & $0.1122$ &\\
$0.4$  & $0.0218$ & $0.0560$ & $0.0754$ & $0.0873$ & $0.1053$ & $0.1136$ &\\
$0.6$  & $0.0247$ & $0.0596$ & $0.0785$ & $0.0900$ & $0.1071$ & $0.1150$ &\\
$0.8$  & $0.0280$ & $0.0633$ & $0.0818$ & $0.0928$ & $0.1090$ & $0.1164$ &\\
$0.99$ & $0.0315$ & $0.0672$ & $0.0850$ & $0.0956$ & $0.1109$ & $0.1177$ &\\
\hline
\end{tabular}
\end{minipage}
\\[0.5cm]
\begin{minipage}[h]{4.2in}
\captionof{figure}{Graph of $f_{3}$ (see Theorem \ref{s4thm3}) and the functions whose roots are the Bohr radius (page 57 of \cite{aliagh2022}) for $g.$}
\vspace{-0.45cm}
\includegraphics[width=10.9cm, height=7.3cm]{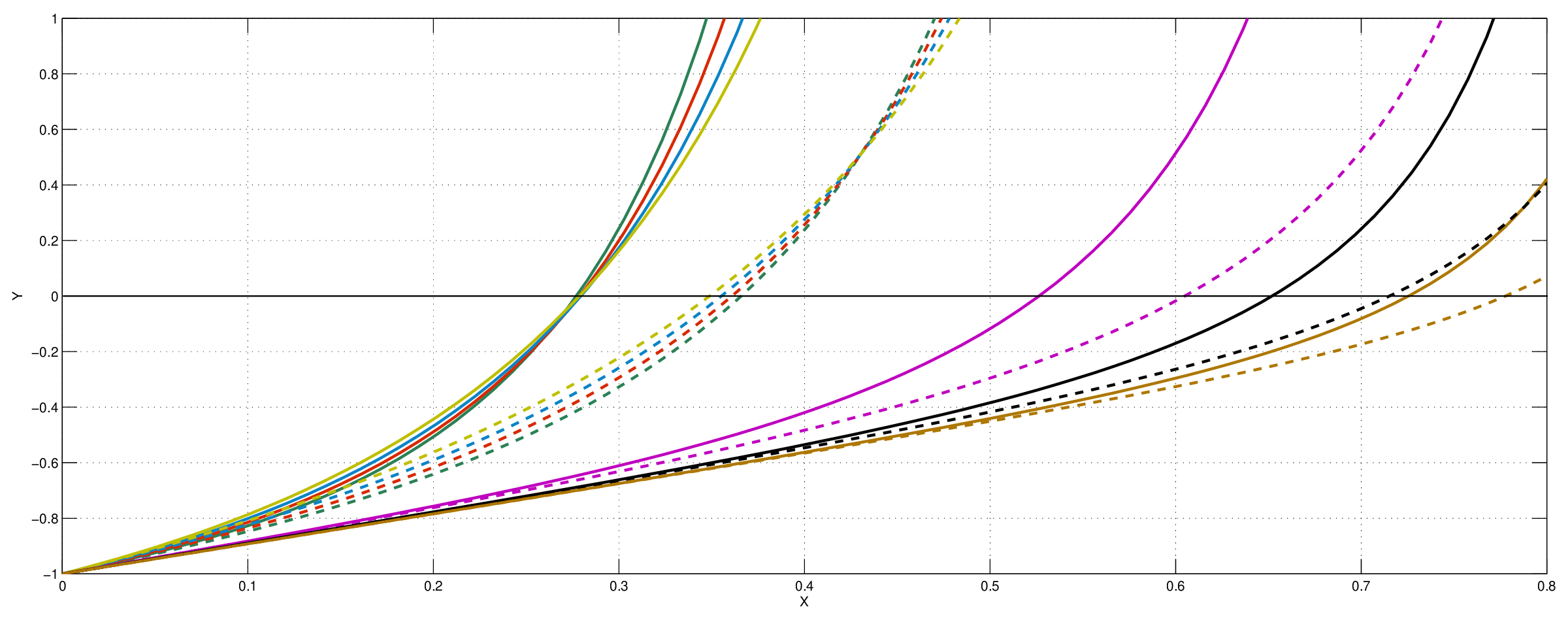}
\label{fig3}
\end{minipage}
\hspace{1cm}
\begin{minipage}[h]{4.2in}
\captionof{figure}{Graph of $f_{4}$ (see Theorem \ref{s4thm3})}
\vspace{-0.45cm}
\includegraphics[width=10.9cm, height=7.3cm]{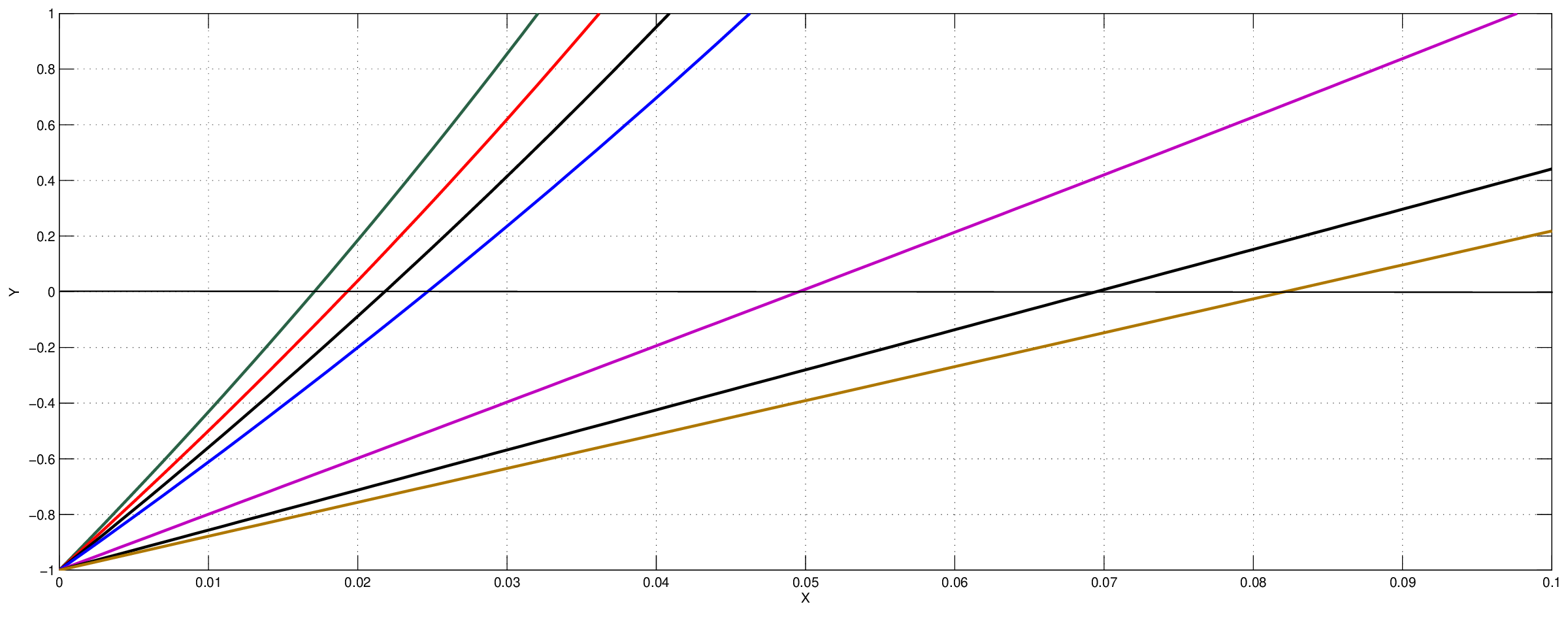}
\label{fig4}
\end{minipage}
\end{sidewaystable}

\begin{sidewaystable}
\begin{minipage}[h]{4.3in}
\fontsize{0.3cm}{0cm}
\caption{Improved Bohr radius $r_{5}$ obtained in Theorem \ref{s4thm4}}
\label{tab5}
\begin{tabular}{|c|ccccccc|c}
\hline
\backslashbox{$\alpha \downarrow$}{$k \rightarrow$}& $1$& 2&3&4&7&10&\\
\hline
$0$    & $0.0592$ & $0.1779$ & $0.2539$ & $0.3020$ & $0.3756$ & $0.4049$ &\\
$0.2$  & $0.0667$ & $0.1891$ & $0.2644$ & $0.3113$ & $0.3823$ & $0.4144$ &\\
$0.4$  & $0.0754$ & $0.2010$ & $0.2754$ & $0.3210$ & $0.3890$ & $0.4195$ &\\
$0.6$  & $0.0853$ & $0.2138$ & $0.2870$ & $0.3311$ & $0.3959$ & $0.4247$ &\\
$0.8$  & $0.0969$ & $0.2276$ & $0.2991$ & $0.3414$ & $0.4029$ & $0.4299$ &\\
$0.99$ & $0.1097$ & $0.2416$ & $0.3111$ & $0.3516$ & $0.4097$ & $0.4350$ &\\
\hline
\end{tabular}
\end{minipage}
\hspace{1cm}
\begin{minipage}[h]{4.3in}
\fontsize{0.3cm}{0cm}
\caption{Bohr type radius $r_{6}$ obtained in Theorem \ref{s4thm5}}
\label{tab6}
\begin{tabular}{|c|ccccccc|c}
\hline
\backslashbox{$\alpha \downarrow$}{$k \rightarrow$}& $1$& 2&3&4&7&10&\\
\hline
$0$   & $0.0505$ & $0.1665$ & $0.2460$ & $0.2974$ & $0.3749$ & $0.4093$ &\\
$0.2$ & $0.0570$ & $0.1769$ & $0.2562$ & $0.3066$ & $0.3816$ & $0.4143$ &\\
$0.4$ & $0.0644$ & $0.1883$ & $0.2671$ & $0.3163$ & $0.3883$ & $0.4194$ &\\
$0.6$ & $0.0732$ & $0.2010$ & $0.2787$ & $0.3264$ & $0.3953$ & $0.4246$ &\\
$0.8$ & $0.0838$ & $0.2147$ & $0.2912$ & $0.3371$ & $0.4023$ & $0.4299$ &\\
$0.99$& $0.0961$ & $0.2295$ & $0.3039$ & $0.3478$ & $0.4092$ & $0.4349$ &\\
\hline
\end{tabular}
\end{minipage}
\\[0.2cm]
\begin{minipage}[h]{4.2in}
\captionof{figure}{Graph of $f_{5}$ (see Theorem \ref{s4thm4})}
\vspace{-0.4cm}
\includegraphics[width=11cm, height=8cm]{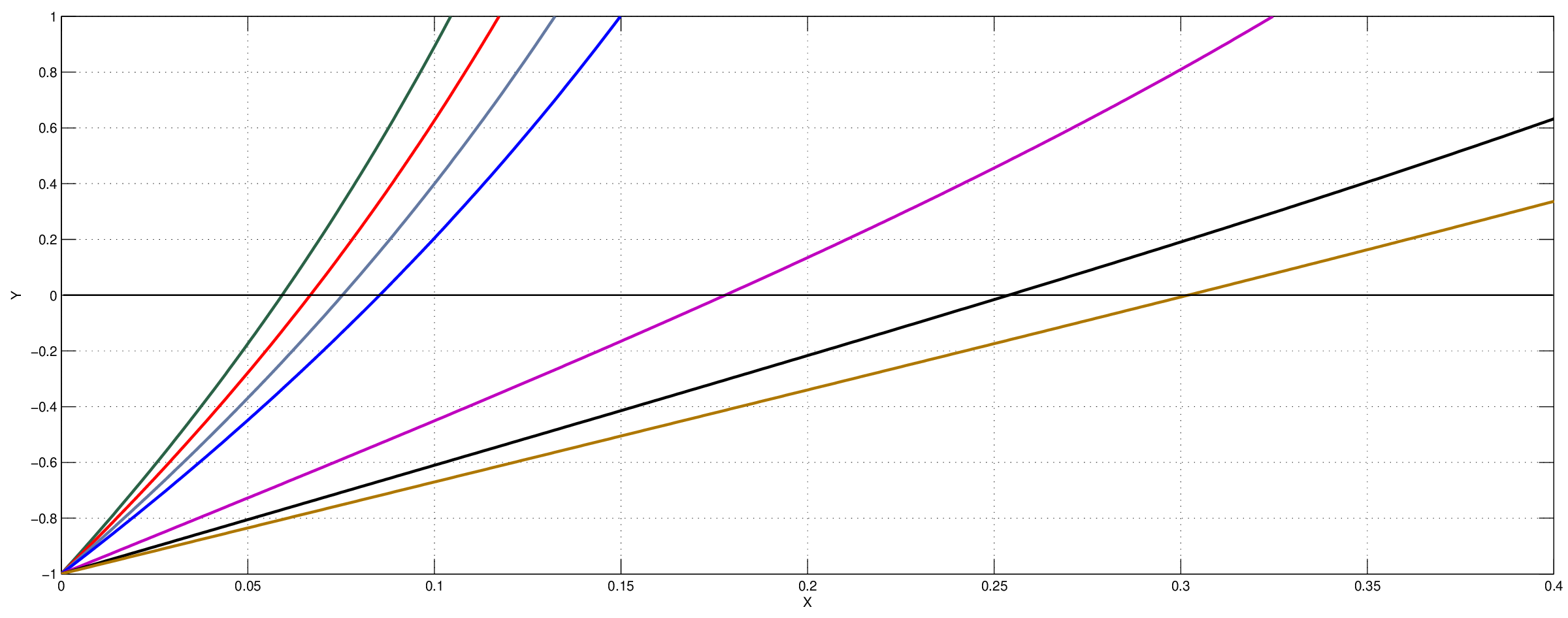}
\label{fig5}
\end{minipage}
\hspace{1cm}
\begin{minipage}[h]{4.2in}
\captionof{figure}{Graph of $f_{6}$ (see Theorem \ref{s4thm5})}
\vspace{-0.48cm}
\includegraphics[width=11cm, height=8cm]{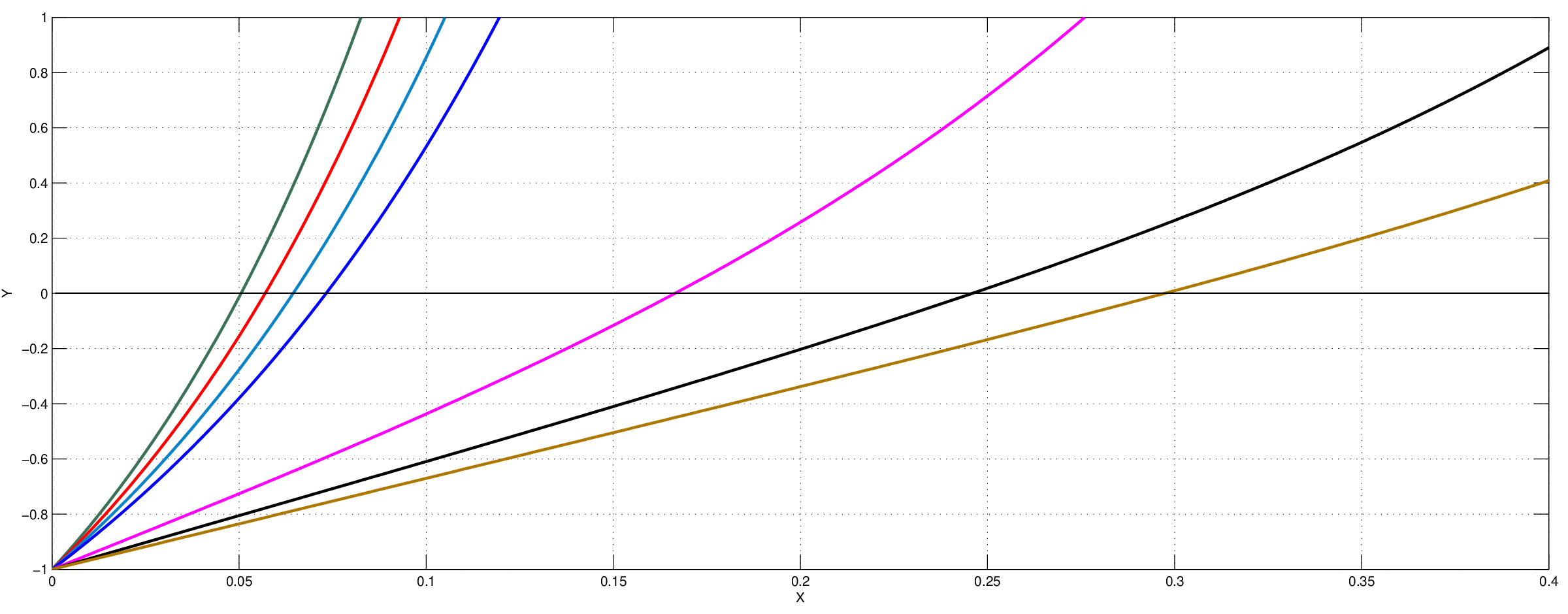}
\label{fig6}
\end{minipage}
\end{sidewaystable}

In view of the above tables and graphs, we have the following observations. The improved Bohr radii obtained for $f$, $h$ and $g$ are smaller than the corresponding Bohr radii. It is pointed out that the Bohr radius increases with increases in the number of folds in the functions. In this regard, another important point is that the functions that provide the Bohr radius and the improved Bohr radius are strictly increasing in $(0,1)$ and move towards $\infty$ as $r \rightarrow 1$, that is, the line $x=1$ is a tangent for these functions. 

\section{Pre-Schwarzian and Schwarzian derivative}\label{schwarzian}
For an analytic univalent function $\kappa : \mathbb{D}\rightarrow \mathbb{C}$, the classical pre-Schwarzian and Schwarzian derivatives are respectively defined as (cf. \cite{dur1983})
\begin{equation*}\label{s3e1}
P_{\kappa}(z)=\left(\log(\kappa^{\prime}(z))\right)^{\prime}=\dfrac{\kappa^{\prime \prime}(z)}{\kappa^{\prime}(z)}
\end{equation*}
and
\begin{equation*}\label{s3e2}
S_{\kappa}(z)=(P_{\kappa}(z))^{\prime}-\dfrac{1}{2}(P_{\kappa}(z))^{2}=\left(\dfrac{\kappa^{\prime \prime}(z)}{\kappa^{\prime}(z)}\right)^{\prime}-\dfrac{1}{2}\left(\dfrac{\kappa^{\prime \prime}(z)}{\kappa^{\prime}(z)}\right)^{2}.
\end{equation*}
For the case of complex-valued harmonic mappings, this theory was first introduced in \cite{chudur2003} by Chuaqui et al. and further investigated in \cite{chudur2007, chudur2008, hermar2015}. Mao and Ponnusamy \cite{maopon2013} studied the theory of Schwarzian derivative for non-vanishing logharmonic mappings and studied different conditions for the Schwarzian derivative to be analytic. Liu and Ponnusamy in \cite{liupon2018} modified the definitions of pre-Schwarzian and Schwarzian derivatives as in \cite{maopon2013} and concluded that the new definitions preserve the standard properties of the classical Schwarzian derivative. 

In this section, we introduce the pre-Schwarzian and Schwarzian derivatives for univalent logharmonic mappings $f$ of the form $f(z)=zh(z)\overline{g(z)}$ and show that the definitions preserve some classical properties as in the analytic univalent case.

We define the pre-Schwarzian and Schwarzian derivatives with the assistance of Jacobian. The pre-Schwarzian derivative of univalent logharmonic mappings of the form $f(z)=zh(z)\overline{g(z)}$ is defined as:
\begin{equation}\label{s3e3}
\begin{split}
P_{f}(z)=\left(\log J_{f}(z)\right)_{z}&=\dfrac{\partial}{\partial z}\left( \log f_{z}(z)+\log \overline{f_{z}(z)}+\log (1-|\omega(z)|^{2})\right) \nonumber \\
&=\dfrac{2h^{\prime}(z)+zh^{\prime \prime}(z)}{h(z)+zh^{\prime}(z)}+\dfrac{g^{\prime}(z)}{g(z)}-\dfrac{\omega^{\prime}(z)\overline{\omega(z)}}{1-|\omega(z)|^{2}} \nonumber , \qquad z\in \mathbb{D}
\end{split}
\end{equation}  
and the Schwarzian derivative of $f(z)=zh(z)\overline{g(z)}$ is
\begin{equation}\label{s3e4}
\begin{split}
S_{f}(z)&=\left(P_{f}(z)\right)^{\prime}-\dfrac{1}{2}\left(P_{f}(z)\right)^{2} \nonumber \\
&= \left( \dfrac{2h^{\prime}(z)+zh^{\prime \prime}(z)}{h(z)+zh^{\prime}(z)}+\dfrac{g^{\prime}(z)}{g(z)}-\dfrac{\omega^{\prime}(z)\overline{\omega(z)}}{1-|\omega(z)|^{2}} \right)^{\prime}-\dfrac{1}{2}\left(\dfrac{2h^{\prime}(z)+zh^{\prime \prime}(z)}{h(z)+zh^{\prime}(z)}+\dfrac{g^{\prime}(z)}{g(z)}-\dfrac{\omega^{\prime}(z)\overline{\omega(z)}}{1-|\omega(z)|^{2}}\right)^{2} \nonumber \\
&=\left(\dfrac{2h^{\prime}(z)+zh^{\prime \prime}(z)}{h(z)+zh^{\prime}(z)}+\dfrac{g^{\prime}(z)}{g(z)}\right)^{\prime}-\dfrac{1}{2}\left(\dfrac{2h^{\prime}(z)+zh^{\prime \prime}(z)}{h(z)+zh^{\prime}(z)}\right)^{2}-\dfrac{\omega^{\prime \prime}(z)\overline{\omega(z)}}{1-|\omega(z)|^{2}}-\dfrac{3}{2}\left( \dfrac{\omega^{\prime}(z)\overline{\omega(z)}}{1-|\omega(z)|^{2}}\right)^{2} \nonumber \\
& \hphantom{\left(\dfrac{2h^{\prime}(z)+zh^{\prime \prime}(z)}{h(z)+zh^{\prime}(z)}+\dfrac{g^{\prime}(z)}{g(z)}\right)^{\prime}-\dfrac{1}{2}}{}+\left(\dfrac{2h^{\prime}(z)+zh^{\prime \prime}(z)}{h(z)+zh^{\prime}(z)}+\dfrac{g^{\prime}(z)}{g(z)}\right)\left(\dfrac{\omega^{\prime}(z)\overline{\omega(z)}}{1-|\omega(z)|^{2}}\right), \qquad z \in \mathbb{D}.
\end{split}
\end{equation}
The pre-Schwarzian and Schwarzian derivatives of logharmonic mapping $f$ obey the same chain rule property as in the analytic case. The following theorem narrates detail.
\begin{theorem}
Let $f$ be a logharmonic mapping of the form$f(z)=zh(z)\overline{g(z)}$ and $\phi$ be a univalent analytic function, then
\begin{itemize}
\item[(i)] $P_{f \circ \phi}(z)=\left(P_{f} \circ \phi(z)\right)\cdot \phi ^{\prime}(z)+P_{\phi}(z)$
\item[(ii)] $S_{f \circ \phi}(z)=\left(S_{f} \circ \phi(z)\right)\cdot (\phi ^{\prime}(z))^{2}+S_{\phi}(z)$
\end{itemize}
\end{theorem}
\begin{proof}
\itshape{(i)} From the definition of Jacobian of logharmonic mappings, we have
\begin{align}\label{s3l1e1}
J_{f \circ \phi}(z)=|(f \circ \phi)_{z}(z)|^{2}\left(1-|(\omega \circ \phi)(z)|^{2}\right).
\end{align}
The logarithmic derivative of \eqref{s3l1e1} with respect to $z$ gives
\begin{eqnarray}
\dfrac{\partial}{\partial z}\left( \log \left( J_{f \circ \phi}(z)\right) \right) &=& \dfrac{\partial}{\partial z}\left( \log(f\circ \phi)_{z}(z)+\log(\overline{(f\circ \phi)}_{z})+\log (1-|(\omega \circ \phi)(z)|^{2})\right) \nonumber\\
 &=& \dfrac{\partial}{\partial z}\left( \log \phi ^{\prime}(z)+\log \left( h(\phi(z))\overline{g(\phi(z))}+\phi(z)h^{\prime}(\phi(z))\overline{g(\phi(z))}\right)\right)\nonumber\\
 && + \dfrac{\partial}{\partial z}\left(\log \overline{\phi^{\prime}(z)}+\log \left( \overline{h(\phi(z))}g(\phi(z))+\overline{\phi(z)}\overline{h^{\prime}(\phi(z))}g(\phi(z))\right)\right)\nonumber \\ 
&& \hphantom{+\dfrac{\partial}{\partial z}\left(\log \overline{\phi^{\prime}(z)}+\log \left( \overline{h(\phi(z))}\right)\right)}{} + \dfrac{\partial}{\partial z}\left(1-|\omega(f(\phi(z)))|^{2}\right)\nonumber\\
&=& \dfrac{\phi^{\prime \prime}(z)}{\phi^{\prime}(z)}+\phi^{\prime}(z)\left( \dfrac{2h^{\prime}(\phi(z))+\phi(z)h^{\prime \prime}(\phi(z))}{h(\phi (z))+\phi(z)h^{\prime}(\phi(z))}+\dfrac{g^{\prime}(\phi(z))}{g(\phi(z))}\right)\nonumber\\
&& \hphantom{\dfrac{\phi^{\prime \prime}(z)}{\phi^{\prime}(z)}+\phi^{\prime}(z)\left( \dfrac{2h^{\prime}(\phi(z))+\phi(z)h^{\prime \prime}(\phi(z))}{h(\phi (z))+\phi(z)h^{\prime}(\phi(z))}\right)}{} -\dfrac{\phi^{\prime}(z)\omega(\phi(z))\overline{\omega(\phi(z))}}{1-|\omega(\phi(z))|^{2}} \nonumber\\
 &=& \left(P_{f} \circ \phi(z)\right)\cdot \phi ^{\prime}(z)+P_{\phi}(z)). \nonumber
\end{eqnarray}
\itshape{(ii)} From the definition of Schwarzian derivative, we have
\begin{align*}
S_{f\circ \phi}(z)&=\left(P_{f\circ \phi}(z)\right)^{\prime}-\dfrac{1}{2}\left(P_{f\circ \phi}(z)\right)^{2} \\
&=\left(\left(P_{f} \circ \phi(z)\right)\cdot \phi ^{\prime}(z)+P_{\phi}(z)\right)^{\prime}-\dfrac{1}{2}\left( \left(P_{f} \circ \phi(z)\right)\cdot \phi ^{\prime}(z)+P_{\phi}(z)\right)^{2}\\ 
&=\left( (P_{f}(\phi(z)))^{\prime}-\dfrac{1}{2}(P_{f}(\phi(z))^{2}\right)(\phi^{\prime}(z))^{2}+\left(\dfrac{\phi^{\prime \prime}(z)}{\phi^{\prime}(z)}\right)^{\prime}-\dfrac{1}{2}\left(\dfrac{\phi^{\prime \prime}(z)}{\phi^{\prime}(z)}\right)^{2} \\
&=\left(S_{f} \circ \phi(z)\right)\cdot (\phi ^{\prime}(z))^{2}+S_{\phi}(z).
\end{align*}
\end{proof}
The next theorem is a direct consequence of Theorem $5.1$ of \cite{liupon2018}.
\begin{theorem}
Let $f$ be a univalent logharmonic mapping bearing the form $f(z)=zh(z)\overline{g(z)}$ with dilatation $\omega (z)$, then $P_{f}$ is harmonic if and only if $\omega$ is constant.
\end{theorem}

\section*{Declarations} 
\noindent{\bf Conflict of interest.}~ Both authors have no conflict of interest.\\
\noindent { \bf Author contributions.}~ The authors contributed equally to this article.\\
\noindent { \bf Funding.}~The first author is supported by the DST-INSPIRE Fellowship, Government of India, INSPIRE code IF190766, and the second author acknowledges the support from the OSHEC of the OURIIP Seed Fund of the Government of Odisha in India, sanction order No. 1040/69/OSHEC.

\end{document}